\documentclass{amsart}
\usepackage{amssymb, enumerate}
\usepackage{amsrefs}
\usepackage{graphicx}

\theoremstyle{plain}
\newtheorem{theorem}{Theorem}
\newtheorem*{menger}{Menger's Theorem}
\newtheorem*{thmJohnuws}{Theorem \ref{Johnuws}}
\newtheorem*{thmcdim}{Theorem \ref{cdim}}
\newtheorem*{propuwsblowup}{Proposition \ref{uwsblowup}}
\newtheorem{lemma}[theorem]{Lemma}

\newtheorem{prop}[theorem]{Proposition}
\newtheorem{corollary}[theorem]{Corollary}

\theoremstyle{definition}
\newtheorem{definition}[theorem]{Definition}

\theoremstyle{remark}
\newtheorem{remark}[theorem]{Remark}
\newtheorem{question}{Question}

\newcommand{\lp}{\left(}

\newcommand{\rp}{\right)}
\newcommand{\e}{\epsilon}
\newcommand{\R}{\mathbb{R}}
\newcommand{\C}{\mathbb{C}}

\newcommand{\Sp}{\mathbb{S}}
\newcommand{\D}{\mathbb{D}}

\newcommand{\N}{\mathbb{N}}
\newcommand{\cl}{\overline}
\newcommand{\ti}{\textit}

\let\mod\relax

\DeclareMathOperator{\dist}{\textup{\text{dist}}}
\DeclareMathOperator{\diam}{\textup{\text{diam}}}

\DeclareMathOperator{\mod}{\textup{\text{mod}}}
\DeclareMathOperator{\SLE}{\textup{\text{SLE}}}

\DeclareMathOperator{\Cdim}{\textup{Cdim}}
\DeclareMathOperator{\AR}{\textup{AR}}

\numberwithin{equation}{section}
\numberwithin{theorem}{section}

\begin{document}

\title{Conformal Dimension and Boundaries of Planar Domains}
\author{Kyle Kinneberg}
\address{Department of Mathematics, Rice University, 6100 Main St., Houston TX 77005}
\email{kyle.kinneberg@rice.edu}

\date{\today}

\maketitle

\begin{abstract}
Building off of techniques that were recently developed by M. Carrasco, S. Keith, and B. Kleiner to study the conformal dimension of boundaries of hyperbolic groups, we prove that uniformly perfect boundaries of John domains in $\hat{\C}$ have conformal dimension equal to 0 or 1. Our proof uses a discretized version of Carrasco's ``uniformly well-spread cut point" condition, which we call the discrete UWS property, that is well-suited to deal with metric spaces that are not linearly connected. More specifically, we prove that boundaries of John domains have the discrete UWS property and that any compact, doubling, uniformly perfect metric space with the discrete UWS property has conformal dimension equal to 0 or 1. In addition, we establish other geometric properties of metric spaces with the discrete UWS property, including connectivity properties of their weak tangents.
\end{abstract}

\section{Introduction} \label{intro}

Let $\Omega \subset \hat{\C}$ be a domain (a connected, open set) in the Riemann sphere. We say that $\Omega$ is a quasi-disk if it is the image of the unit disk $\D$ under a quasiconformal homeomorphism of $\hat{\C}$. In this case, $\partial \Omega$ is a Jordan curve that is quasisymmetrically equivalent to the unit circle $\Sp^1$. From the viewpoint of quasiconformal geometry, and the viewpoint we adopt in this paper, quasi-disks and quasi-circles form extremely nice classes of metric objects.

Here we are interested in two other classes of domains in $\hat{\C}$: John domains and H\"older domains. Informally, a John domain is one in which every point can be joined to a common base-point by a ``twisted cone" with uniform vertex angle. Then, again informally, a H\"older domain is one in which the quasi-hyperbolic metric grows in roughly the same way that it does in a John domain. Every quasi-disk is a John domain, and every John domain is a H\"older domain, but these are both proper containments (even in the simply connected setting).

John domains and H\"older domains arise naturally in conformal dynamics, both from iteration of rational maps and from Kleinian groups. For example, if $p(z)$ is polynomial of degree at least $2$, then all of the Fatou components are John domains if and only if $p(z)$ is semi-hyperbolic \cite{CJY}. An analogous statement about John domains and Kleinian groups is proved in \cite{McM}. More broadly, if $q(z)$ is a rational map that satisfies the Collet--Eckmann expansion condition, then all Fatou components are H\"older domains \cite{GS} (and \cite{Prz} gives a partial converse to this). We note that the latter examples are, in some senses, typical: almost every external ray to the Mandelbrot set lands on a value $c$ for which $q(z) = z^2 +c$ satisfies the Collet--Eckmann condition \cite{Smi}.

H\"older domains also appear in some explicitly stochastic settings, particularly those connected to $\SLE$. For example, the unbounded complementary domain of an $\SLE_\kappa$ trace is, almost surely, a H\"older domain when $\kappa \neq 4$ \cite{RS}. In a similar way, the interior domains for the random Jordan curves constructed in \cite{AJKS} are, almost surely, H\"older domains. Here we should mention that the H\"older domains appearing in these contexts are, almost surely, not John domains. At the same time, due to similarities in the behavior of the quasi-hyperbolic metric, we think of John domains as deterministic ``toy examples" of the stochastic domains. 

In all three realms mentioned above (rational dynamics, Kleinian groups, and stochastic planar processes), there are interesting questions concerning quasiconformal equivalence. When are two Kleinian groups quasiconformally conjugate? When are two Julia sets quasisymmetrically equivalent? Are two independent samples of $\SLE_{\kappa}$ traces quasiconformally equivalent almost surely? There is a significant body of work related to the first two questions (classical results include \cite{Tuk} and \cite{MSS}; for further discussion see \cite[Section 6]{Kap} and \cite[Section 8.3]{MT10}), and the third question was recently asked by C. McMullen.

In light of these questions, we are motivated to study quasisymmetric invariants of the relevant objects. A prominent invariant coming from hyperbolic geometry is the conformal dimension. Originally introduced by Pansu \cite{Pan} to study visual boundaries of rank-one symmetric spaces, the conformal dimension can be defined for any metric space $(X,d)$ and is denoted by $\Cdim(X)$. More accurately, we will concentrate on a related invariant, the Ahlfors-regular conformal dimension, denoted $\Cdim_{\AR}(X)$. These have become important tools in the study of hyperbolic groups and, more generally, in the analysis of metric spaces \cite{MT10}.

Returning to the setting of planar domains, in this paper we focus primarily on questions related to the conformal dimension of $\partial \Omega$, where $\Omega$ is a John domain. Indeed, if $\Omega$ is a quasi-disk, it follows immediately from the definitions that the conformal dimension of $\partial \Omega$ is equal to 1. On the other hand, we suspect that the conformal dimension of boundaries of H\"older domains can take any value strictly between 1 and 2. Thus, John domains seem to be the appropriate class to investigate. We will prove the following theorem.

\begin{theorem} \label{John1}
If $\Omega \subset \hat{\C}$ is a John domain with $\partial \Omega$ uniformly perfect, then $\Cdim_{\AR}(\partial \Omega) \in \{0,1\}$. It is is equal to $0$ if and only if $\partial \Omega$ is uniformly disconnected.
\end{theorem}

We should note that, in general, $\Cdim(X) \leq \Cdim_{\AR}(X)$. Thus, we also obtain $\Cdim(\partial \Omega) \in \{0,1\}$, as the conformal dimension cannot assume values in $(0,1)$ by the results in \cite{Kov}. If we assume in addition that $\partial \Omega$ is connected and not a singleton, then this gives $\Cdim(\partial \Omega) = \Cdim_{\AR}(\partial \Omega)=1$. 

Before moving on, let us mention a corollary that deals with the quasiconformal geometry of certain Julia sets. We noted above that if $p(z)$ is a semi-hyperbolic polynomial on $\hat{\C}$ of degree at least 2, then its unbounded Fatou component is a John domain \cite{CJY}. The Julia set $J(p)$ of $p$ is the boundary of this component, so Theorem \ref{John1} implies the following.

\begin{corollary}
If $p(z)$ is a semi-hyperbolic polynomial of degree at least 2 with $J(p)$ connected, then $\Cdim(J(p)) = \Cdim_{\AR}(J(p)) = 1$.
\end{corollary}

We note that this statement is proven in M. Carrasco's PhD thesis \cite[Corollary 3.3]{CarPhD}, using the dynamics of $p$ on $J(p)$. One could therefore view Theorem \ref{John1} as a non-dynamical extension of this fact.

To establish Theorem \ref{John1}, we will introduce the notion of uniformly well-spread discrete cut points (which we call the \ti{discrete UWS property}). This is a direct discretization of the notion of uniformly well-spread cut points in a metric space (the \ti{UWS property}), which was recently introduced by M. Carrasco as a sufficient condition for a compact, doubling, linearly connected metric space to have conformal dimension equal to 1 \cite{CP14}. His motivation came from boundaries of one-ended hyperbolic groups, where the linear connectivity condition is automatically satisfied (but the boundary need not be planar). For us, $\partial \Omega$ is planar and, generally, is topologically uncomplicated (namely, $\Omega$ is connected), but linear connectivity is almost never satisfied. For example, if $\partial \Omega$ were a linearly connected Jordan curve, then it would be a quasi-circle, and the conformal dimension of $\partial \Omega$ would trivially be equal to 1. However, a simple discretization of the original UWS property works well for the objects we consider. The proof of Theorem \ref{John1} breaks into two parts.

\begin{theorem} \label{Johnuws}
If $\Omega \subset \hat{\C}$ is an $L$-John domain, then $\partial \Omega$ has the discrete UWS property, with constant depending only on $L$.
\end{theorem}

\begin{theorem} \label{cdim}
If $X$ is a compact, doubling, uniformly perfect metric space and has the discrete UWS property, then $\Cdim_{\AR}(X) \in \{0,1\}$. It is equal to $0$ if and only if $X$ is uniformly disconnected.
\end{theorem}

The discrete UWS property is a central ingredient in the proof of our main result. We introduce it in Section \ref{duws} and establish some basic properties before proving Theorem \ref{cdim}. In particular, we discuss its relationship to the standard UWS property and its quasisymmetric invariance in the linearly connected setting.

In Section \ref{Johnsec} we prove Theorem \ref{Johnuws}. In many ways, this is the heart of the paper, and we think it justifies our subsequent study of the discrete UWS property. In Section \ref{WTsec}, we turn our attention to the infinitesimal geometry of metric spaces with the discrete UWS property by studying their weak tangents. For example, we will establish the following result.

\begin{prop} \label{uwsblowup}
Let $X$ be a complete, connected, doubling metric space that has the discrete UWS property with constant $C$. Then every weak tangent of $X$ has at most $N$ connected components, where $N$ depends only on $C$ and on the doubling constant of $X$.
\end{prop}

In particular, if $\Omega$ is a John domain with $\partial \Omega$ connected, then every weak tangent of $\partial \Omega$ has a uniformly bounded number of connected components (Theorem \ref{WTjohn}). One should compare this to the following asymptotic characterization of quasi-circles: a Jordan curve is a quasi-circle if and only if every weak tangent is connected (see Theorem \ref{metriccircle}).

Finally, in Section \ref{Holdersec}, we discuss some problems related to the conformal dimension of boundaries of H\"older domains. We give an example of a H\"older domain in $\C$ whose boundary has Hausdorff dimension equal to 1 but Ahlfors-regular conformal dimension equal to 2. We end with some questions about whether this behavior is typical for the complementary components of $\SLE$ traces.

\subsection*{Acknowledgements}

The author thanks Mario Bonk for many good discussions about the topics that appear in this paper. In particular, Theorem \ref{metriccircle} was motivated by a question of his, and Lemma \ref{hops} grew out of his suggested approach to that question. This provided the kernel from which the other arguments and results in this paper eventually grew. Additional thanks go to Steffen Rhode and Huy Tran for helpful discussions about trees and $\SLE$ curves, to Hrant Hakobyan, John Mackay, Jeremy Tyson, and Dimitrios Ntalampekos for feedback on an earlier version of this paper, and to the anonymous referee for helpful suggestions.

\section{Some definitions and background}

Let us recall some standard notation and definitions in metric geometry. A metric space $(X,d)$ is doubling if there is a constant $C$ for which every ball in $X$ of radius $r$ can be covered by at most $C$ balls of radius $r/2$. Note that a doubling metric space is totally bounded, so every complete and doubling metric space is proper: closed balls are compact. 

We say that $X$ is uniformly perfect if there is $c>0$ such that the closed annulus $\cl{B}(x,r) \backslash B(x,cr) \neq \emptyset$ for each $x \in X$ and $0<r \leq \diam(X)$. For example, if $X$ is connected and has at least two points, then it is uniformly perfect. A much stronger condition is linear connectivity. For $\lambda \geq 1$, we say that $X$ is $\lambda$-linearly connected if for any two points $x,y \in X$ there is a compact connected set $E \subset X$ with $x,y \in E$ and $\diam(E) \leq \lambda d(x,y)$.

Following standard notation, we will use $\dim_H(X)$ to denote the Hausdorff dimension of $X$. For $Q>0$, the metric space $X$ is said to be Ahlfors $Q$-regular if it supports a Borel regular measure $\mu$ such that
$$C^{-1} r^Q \leq \mu(B(x,r)) \leq Cr^Q$$
for all balls in $X$ of radius $0<r\leq \diam(X)$, where $C$ is a uniform constant. In this case, $\dim_H(X) = Q$. It is not difficult to see that an Ahlfors regular metric space is necessarily doubling and uniformly perfect.

A finite sequence of points $x_0,x_1,\ldots,x_\ell$ in $X$ is called a discrete $\delta$-path if $d(x_i,x_{i-1}) \leq \delta$ for each $1\leq i \leq \ell$. Abusing standard terminology, we say that such a discrete path joins $x_0$ and $x_\ell$. The metric space $X$ is uniformly disconnected if there is $\e >0$ such that, for each $x,y \in X$ distinct, there is no discrete $\e d(x,y)$-path joining $x$ and $y$.

Now, we turn our attention to sets in the Riemann sphere. When working in this setting, we will always use the spherical metric on $\hat{\C}$, though on occasion we will still denote it by $| x- y|$. Let $\Omega \subset \hat{\C}$ be a domain. For $z \in \Omega$, let $\delta_{\Omega}(z) = \dist(z,\partial \Omega)$ be the distance from $z$ to the boundary of $\Omega$, measured of course in the spherical metric. If $\alpha$ is an arc in $\Omega$, we will use $\alpha[z,z']$ to denote the closed sub-arc joining two points $z,z' \in \alpha$.

Let $\alpha \subset \Omega$ be an arc with endpoints $z_0$ and $z_1$, and let $L \geq 1$. We say that $\alpha$ is an $L$-John arc with base-point $z_0$ if
$$\diam(\alpha[z,z_1]) \leq L\delta_{\Omega}(z)$$
for each $z \in \alpha$. Often we refer to the other endpoint, $z_1$, as the tip of $\alpha$. Geometrically, a John arc is the core of a twisted cone in $\Omega$, namely the union of all balls $B(z, \diam(\alpha[z,z_1])/L)$ for $z \in \alpha$. 

\begin{definition}
A domain $\Omega \subset \hat{\C}$ is called an $L$-John domain if there is $z_0 \in \Omega$ such that every $z \in \Omega$ can be joined to $z_0$ by an $L$-John arc with base-point $z_0$ and tip $z$.
\end{definition}

If $\Omega$ is a John domain with respect to some base-point, then it is also a John domain with respect to any other base-point, although the constant may change. The John condition functions, in some ways, as a one-sided quasi-disk condition. For example, if $\partial \Omega$ is a Jordan curve and both complementary components are John domains, then $\partial \Omega$ is a quasi-circle (and the complementary components are, in fact, quasi-disks) \cite[Theorem 9.3]{NV}. More generally, a simply connected domain $\Omega \subset \hat{\C}$ is a John domain if and only if $\hat{\C} \backslash \Omega$ is linearly connected \cite[Theorem 4.5]{NV}. The main idea behind these properties is that John domains may have inward-pointing spikes and bubbles but not outward-pointing ones.

A more general class of domains are the H\"older domains, which verify a logarithmic growth condition on the quasi-hyperbolic metric. Namely, for $z,z' \in \Omega$, the quasi-hyperbolic distance is
$$\rho(z,z') = \inf_{\gamma} \int_{\gamma} \frac{ds}{\delta_{\Omega}(w)},$$
where the infimum is taken over rectifiable paths in $\Omega$ that join $z$ and $z'$. Once again, we remark that distances are computed in the spherical metric. This defines a complete metric on $\Omega$ as long as $\Omega \neq \hat{\C}$. In fact, $\rho$ is bi-Lipschitz equivalent to the hyperbolic metric on $\Omega$ whenever $\Omega$ is simply connected.

We say that $\Omega$ is a H\"older domain if there is a base-point $z_0 \in \Omega$ and constants $C_1,C_2$ such that
$$\rho(z_0,z) \leq C_1 \log \lp \frac{1}{\delta_\Omega(z)} \rp  + C_2$$
for each $z \in \Omega$. This condition is also independent of the base-point, although the constants may differ. It is not difficult to show that every John domain is a H\"older domain, but the converse is not true. When $\Omega$ is simply connected (and not equal to $\hat{\C}$ or $\C$), the H\"older condition is equivalent to H\"older continuity of the conformal map $f\colon \D \rightarrow \Omega$ \cite{BP}. 

Every quasi-disk is a John domain and, therefore, is also a H\"older domain. An important difficulty that we will face in this paper is that boundaries of John domains, unlike boundaries of quasi-disks, need not be linearly connected (again, the John condition allows for inward-pointing cusps).

\subsection{Conformal dimension}
Let $(X,d)$ and $(Y,d')$ be metric spaces. A homeomorphism $f \colon X \rightarrow Y$ is quasisymmetric if there is a control function $\eta \colon [0,\infty) \rightarrow [0,\infty)$, i.e., an increasing homeomorphism, such that
$$\frac{d'(f(x),f(y))}{d'(f(x),f(z))} \leq \eta \lp \frac{d(x,y)}{d(x,z)} \rp$$
for all distinct points $x,y,z \in X$. We should remark that the inverse of a quasisymmetric map is quasisymmetric with control function $1/\eta^{-1}(1/t)$. In the case that $X=Y=\hat{\C}$, the class of quasisymmetric maps is precisely the class of quasiconformal maps, in any of the standard definitions of quasiconformal. For our purposes, one can simply take this to be the definition of a quasiconformal homeomorphism of $\hat{\C}$.

Given a complete and doubling metric space $(X,d)$, we use $\mathcal{J}(X)$ to denote the conformal gauge of $X$: the set of all (isomorphism classes of) metric spaces $Y$ that are quasisymmetrically equivalent to $X$. Obviously, $X \in \mathcal{J}(X)$, so the conformal gauge is always non-empty. If, in addition, $X$ is uniformly perfect, then we use $\mathcal{J}_{\AR}(X)$ to denote the the subset of $\mathcal{J}(X)$ consisting of metric spaces that are Ahlfors regular. In this case, $\mathcal{J}_{\AR}(X)$ is non-empty; see for example \cite[Theorem 14.16]{Hein01}, where it is shown that there is $Y \in \mathcal{J}_{\AR}(X)$ which is a closed subset of some $\R^n$. For the most part, we will restrict ourselves to complete, doubling, uniformly perfect metric spaces.

The conformal dimension of $X$ is defined to be
$$\Cdim(X) = \inf \{ \dim_H(Y) : Y \in \mathcal{J}(X) \},$$
and the Ahlfors-regular conformal dimension is
$$\Cdim_{\AR}(X) = \inf \{ \dim_H(Y) : Y \in \mathcal{J}_{\AR}(X) \}.$$
Note that both quantities are finite (when $X$ is doubling, complete, and uniformly perfect), and they are, by definition, quasisymmetric invariants of $X$. It is also obvious that $\Cdim(X) \leq \Cdim_{\AR}(X)$. We should note that if $X \subset \R^n$ is closed and uniformly perfect, then $\Cdim_{\AR}(X) \leq n$ \cite[Corollary 14.17]{Hein01}. Estimating these quantities (even for self-similar spaces) is typically difficult, though there are large classes of sets where it can be done. Particularly relevant for us is the recent work of M. Carrasco, S. Keith, and B. Kleiner, which characterizes the Ahlfors-regular conformal dimension as a certain critical exponent arising from combinatorial modulus estimates on annuli. To describe this, we must introduce more definitions.

\subsection{Critical exponents and the UWS property}

Let $G = (V,E)$ be a finite graph with vertex set $V$ and edge set $E$. A vertex path in $G$ is a (finite) sequence of vertices for which any two consecutive vertices are joined by an edge. Let $\Gamma$ be a collection of vertex paths in $G$. A weight function $\rho \colon V \rightarrow [0,\infty]$ is admissible for $\Gamma$ if 
$$\sum_{v \in \gamma} \rho(v) \geq 1$$
for each $\gamma \in \Gamma$. The combinatorial $p$-modulus of $\Gamma$ is then defined to be
$$\mod_p(\Gamma, G) = \inf_{\rho} \sum_{v \in V} \rho(v)^p,$$
where the infimum is taken over all weight functions $\rho$ that are admissible for $\Gamma$. This is a combinatorial version of the standard notion of $p$-modulus for path families in metric spaces \cite[Chapter 7]{Hein01}.

Suppose that $X$ is a compact, doubling, and uniformly perfect metric space. Let us form discrete approximations to $X$ as follows. Fix $a >1$ and $\lambda \geq 32$. For each $k \in \N$, let $P_k$ be a maximal $a^{-k}$-separated set in $X$, so that $P_k$ is finite by compactness and 
$$X = \bigcup_{x \in P_k} B(x,a^{-k})$$
by maximality. Now define $G_k$ to be the graph with vertex set $P_k$, where we connect two vertices $x,y \in P_k$ by an edge if 
$$B(x,\lambda a^{-k}) \cap B(y,\lambda a^{-k}) \neq \emptyset.$$
For each fixed $m \in \N$ and $x \in P_m$, we use $\Gamma_k(x)$ to denote the family of vertex paths in $G_{m+k}$ that join $P_{m+k} \cap B(x,a^{-m})$ to $P_{m+k} \cap (X \backslash \cl{B}(x, 2a^{-m}))$. Observe that each vertex path in $\Gamma_k(x)$ forms a discrete $2\lambda a^{-m-k}$-path in the metric space $X$ that ``crosses" the annulus $\cl{B}(x,2a^{-m}) \backslash B(x,a^{-m})$. We then define
$$M_p(k) = \sup_{m \in \N} \sup_{x \in P_m} \mod_p(\Gamma_k(x),G_{m+k})$$
and $M_p = \liminf_{k \rightarrow \infty} M_p(k)$. It is not difficult to see that, for each $k$, the quantity $M_p(k)$ is non-increasing in $p$, and therefore $M_p$ is non-increasing as well. Indeed, any optimal weight function $\rho$ for $\mod_p(\Gamma_k(x),G_{m+k})$, which always exists because $G_k$ is finite, necessarily takes values in $[0,1]$. Consequently, we can define the critical exponent to be $Q_N = \inf \{ p : M_p =0\}$.

\begin{theorem}[Kleiner--Keith \cite{KK}; Carrasco {\cite[Theorem 1.2]{CP13}}]
If $X$ is compact, doubling, and uniformly perfect, then $Q_N = \Cdim_{\AR}(X)$. In particular, the critical exponent does not depend on our choices of $a$, $\lambda$, or the sets $P_k$. 
\end{theorem}

This result was obtained by Kleiner--Keith in unpublished form (see \cite[Corollary 3.7]{BourK} and the subsequent remarks). A detailed proof by Carrasco can be found in \cite{CP13}; our notation closely follows his.

The critical exponent $Q_N$ is essentially a combinatorial quantity: it is calculated by solving a modulus problem on graph approximations to the underlying metric space $X$. However, the vertex paths in these graph approximations might look very different from paths in $X$. For example, the vertex paths in $G_k$ might ``jump" between connected components of $\cl{B}(x,2r) \backslash B(x,r)$, thereby giving vertex paths that do not correspond to actual paths in $X$. One can, however, define a modified critical exponent $Q_X$ that takes into account only those vertex paths coming from true paths in $X$, as in \cite[Section 3.5]{CP13}. In general $Q_X \leq Q_N$, but equality holds when $X$ is linearly connected \cite[Theorem 3.12]{CP13}. Thus, in the linearly connected setting, one can calculate the Ahlfors-regular conformal dimension using path families in $X$.

In light of this, Carrasco introduced the notion of ``uniformly well-spread cut points" in order to study linearly connected metric spaces with Ahlfors-regular conformal dimension equal to 1 \cite{CP14}. A metric space $X$ has the UWS property if there is a constant $C \geq 1$ such that for any $x \in X$ and $0< r < C^{-1}\diam(X)$, there is a finite set $K \subset \cl{B}(x,2r)$, with $\#K \leq C$, for which no connected component of $X \backslash K$ can intersect both $B(x,r)$ and $X \backslash \cl{B}(x,2r)$. More colloquially, this means that we can disconnect the ``complementary regions" of any given annulus $\cl{B}(x,2r) \backslash B(x,r)$ by removing a uniformly bounded number of points. 

\begin{theorem}[Carrasco {\cite[Theorem 1.2]{CP14}}] \label{Car}
If $X$ is compact, connected, doubling, contains at least two points, and has the UWS property, then $Q_X = 1$. In particular, if $X$ is linearly connected, then $\Cdim_{\AR}(X) =\Cdim(X) = 1$.
\end{theorem}

We should remark that a similar unpublished result was obtained earlier by Keith and Kleiner (cf. the discussion following Theorem 1.2 in \cite{CP14}).

As our interest lies mainly with boundaries of John domains, which are typically not linearly connected, the modified critical exponent $Q_X$ will not be helpful. There is an exception, though, which we discuss before moving to the more general setting.

\subsection{An example with quasi-trees} \label{qtree}

Let us say that a compact metric space $T$ is a tree if it has at least two points and any two points are joined by a unique arc (i.e., a unique simple path). If, in addition, $T$ is linearly connected, then we say that it is a quasi-tree. This terminology is not standard, but we think it is clearest for our purposes. For example, a quasi-arc is the simplest type of quasi-tree: it has no branching. 

If $T \subset \hat{\C}$ is a planar quasi-tree and $\Omega = \hat{\C} \backslash T$, then $\Omega$ is a simply connected domain with $\partial \Omega = T$ (indeed, $T$ has no interior). As $\hat{\C} \backslash \Omega = T$ is linearly connected, we know that $\Omega$ is a John domain. Theorem \ref{John1} implies that $\Cdim_{\AR}(T) = 1$, but there is a more direct way to see this using Carrasco's UWS property. In fact, we have the following proposition for general quasi-trees.

\begin{prop}
If $T$ is a doubling quasi-tree, then $T$ has the UWS property. In particular, $\Cdim_{\AR}(T) = \Cdim(T) = 1$.
\end{prop}

\begin{proof}
Fix $x \in T$, and let $0< r < \diam(T)/4$. We first remark that every connected subset of $T$ is arcwise connected; indeed, every connected subset of $T$ is itself a doubling quasi-tree. Thus, to verify the UWS property, it suffices to find $K \subset \cl{B}(x,2r)$, with $\#K$ uniformly bounded, for which every arc from $B(x,r)$ to $T \backslash \cl{B}(x,2r)$ passes through a point in $K$.

To this end, we fix $\gamma_1,\ldots,\gamma_m$ a maximal collection of disjoint arcs in the annulus $\cl{B}(x,2r) \backslash B(x,r)$ with initial point on $\partial B(x,r)$ and terminal point on $\partial B(x,2r)$. Our first claim is that $m$ is uniformly bounded, with bound depending only on the doubling constant, $C$, of $T$ and the constant of linear connectivity, $L$. For each $1\leq i \leq m$, let $x_i \in \gamma_i \cap \partial B(x, r)$ and $y_i \in \gamma_i \cap \partial B(x,2r)$ be the initial point and endpoint of $\gamma_i$. Note that, if $i \neq j$, then it is not possible for $d(x_i,x_j) \leq r/3L$ and $d(y_i,y_j) \leq r/3L$ simultaneously. Otherwise, we could find a path from $x_i$ to $x_j$ and a path from $y_i$ to $y_j$, both of diameter at most $r/3$, and therefore disjoint from each other. The union of these paths with $\gamma_i$ and $\gamma_j$ would contain more than one distinct arc joining $x_i$ and $y_i$, contrary to the fact that $T$ is a tree.

Consider the points $x_1,\ldots,x_m$ and for each $i$, let 
$$N_i = \{ j : d(x_i,x_j) \leq r/6L \}.$$
We observe that there is some $i$ for which $\#N_i \geq cm$, where $c >0$ will shortly be determined and will depend only on $C$ and $L$. Indeed, if $\#N_i < cm$ for each $i$, then we could find a sub-collection $x_{i_1},\ldots,x_{i_\ell}$ that is $r/6L$-separated, with $\ell \geq 1/2c$. As all of the points $x_{i_1},\ldots,x_{i_\ell}$ lie in the ball $\cl{B}(x,r)$, by taking $c$ small enough, depending only on $C$ and $L$, we would contradict the doubling property. Thus, we may fix $i$ for which $\#N_i \geq cm$. For each $j,j' \in N_i$, we have $d(x_j,x_{j'}) \leq r/3L$, so that $\{y_j : j \in N_i\}$ is an $r/3L$-separated set in $\cl{B}(x,2r)$. The doubling property then implies that $\#N_i$ is uniformly bounded, depending only on $C$ and $L$. As $\#N_i \geq cm$, this means that $m$ is uniformly bounded as well.

Let us now form the cut set $K$ using specified points on the arcs $\gamma_1,\ldots,\gamma_m$. For each $1\leq i\leq m$, define the following sets. First, let $E_i$ be a maximal $r/10L$-separated set on $\gamma_i$, so that $\#E_i$ is uniformly bounded, depending only on $C$. Then, let $F_i$ be the set of points $z \in \gamma_i$ for which there is an arc $\alpha$, with initial point $z$ and terminal point on $\partial B(x,r)$, with $\diam(\alpha) \geq r/3$ and $\alpha \cap \gamma_i = \{z\}$. Similarly, let $G_i$ be the set of points $z \in \gamma_i$ for which there is an arc $\alpha$, with initial point $z$ and terminal point on $\partial B(x,2r)$, with $\diam(\alpha) \geq r/3$ and $\alpha \cap \gamma_i = \{z\}$. We claim that both $\#F_i$ and $\#G_i$ are uniformly bounded, depending only on $C$ and $L$.

Let us show that $\#F_i$ is uniformly bounded; the argument for $\#G_i$ is analogous. Let $z_1,\ldots,z_k$ be distinct points in $F_i$, so we wish to give a uniform bound on $k$. Let $\alpha_1,\ldots,\alpha_k$ be the corresponding arcs, and for each $j$, let $w_j$ be the terminal point of $\alpha_j$, which lies on $\partial B(x,r)$. For $1\leq j,j' \leq k$ two distinct indices, let $\gamma_i[z_j,z_{j'}]$ denote the sub-arc of $\gamma_i$ that connects $z_j$ and $z_{j'}$. Then the concatenation of $\alpha_j$, $\gamma_i[z_j,z_{j'}]$, and $\alpha_{j'}$ is an arc that connects $w_j$ to $w_{j'}$. As $T$ is a tree, there is no path connecting $w_j$ and $w_{j'}$ of smaller diameter, so the linear connectivity of $T$ gives
$$r/3 \leq \diam(\alpha_j \cup \gamma_i[z_j,z_{j'}] \cup \alpha_{j'}) \leq Ld(w_j,w_{j'}).$$
Thus, the points $w_1,\ldots,w_k$ form an $r/3L$-separated set in the ball $B(x,2r)$. By doubling, we can conclude that $k$ is bounded uniformly, with bound depending only on $C$ and $L$.

Finally, define the cut set to be $K= \bigcup_{i=1}^m E_i \cup F_i \cup G_i$, so that $\#K$ is bounded uniformly in terms of $C$ and $L$. It remains to show that every arc from $B(x,r)$ to $T \backslash \cl{B}(x,2r)$ must pass through a point in $K$.

Let $\gamma$ be such an arc, so there is an index $i$ for which $\gamma \cap \gamma_i \neq \emptyset$; otherwise, the collection $\gamma_1,\ldots,\gamma_m$ would not be maximal. The fact that $T$ is a tree ensures that $\gamma \cap \gamma_i$ is a connected sub-arc of $\gamma_i$, which we will call $\beta$. Let $s$ and $t$ be the endpoints of $\beta$, so there is a sub-arc, $\beta_0$, of $\gamma$ from $\partial B(x,r)$ to $s$ and a sub-arc, $\beta_1$, of $\gamma$ from $t$ to $\partial B(x,2r)$ such that $\beta_0 \cap \beta = \{s\}$, $\beta_1 \cap \beta = \{t\}$, and $\beta_0 \cap \beta_1 = \emptyset$.

Note that at least one of $\beta_0$, $\beta$, and $\beta_1$ must have diameter at least $r/3$. If $\diam(\beta_0) \geq r/3$, then $s \in F_i$; similarly, if $\diam(\beta_1) \geq r/3$, then $t \in G_i$. In either case, $\gamma$ passes through a point in $K$. Thus, we may assume that $\diam(\beta) \geq r/3$. Let $u \in \beta$ be a point for which $d(s,u) \geq r/10$ and $d(u,t) \geq r/10$. As $\beta$ is a sub-arc of $\gamma_i$, we observe that for each $z \in \gamma_i \backslash \beta$, the sub-arc of $\gamma_i$ from $u$ to $z$ has diameter at least $r/10$, and therefore $d(u,z) \geq r/10L$. As $E_i$ was chosen to be a maximal $r/10L$-separated set in $\gamma_i$, we can conclude that $\beta$ contains at least one point from $E_i$. In particular, $\gamma$ passes through a point in $K$.

This verifies the UWS property for $T$. The fact that $\Cdim_{\AR}(T) = \Cdim(T) = 1$ then follows from Theorem \ref{Car}.
\end{proof}

An immediate question that arises is whether or not the conformal dimension is attained by a metric space in $\mathcal{J}(T)$ or in $\mathcal{J}_{\AR}(T)$. The answer is not always affirmative. In \cite{BT}, C. Bishop and J. Tyson constructed a family of ``antenna" sets, each of which is a planar quasi-tree, but the Hausdorff dimension of any quasisymmetric image is strictly larger than 1. Perhaps more appropriate are questions about whether the conformal dimension of an abstract quasi-tree $T$ can be calculated using planar sets. By the above theorem, $\mathcal{J}_{\AR}(T)$ contains metric spaces with Hausdorff dimension arbitrarily close to 1, so it would suffice to prove a bi-Lipschitz embedding theorem for Ahlfors-regular quasi-trees that have small dimension. A similar statement is proven in \cite{HM} for metric quasi-circles: the authors construct a bi-Lipschitz classification of all quasi-circles, and the representatives are planar when the Assouad dimension is less than 2.

\begin{question}
Does every quasi-tree with Assouad dimension less than 2 admit a bi-Lipschitz embedding into $\C$? If so, are there natural planar representatives for the bi-Lipschitz classes?
\end{question}

\section{The discrete UWS property} \label{duws} 

The class of John domains $\Omega$ for which $\partial \Omega$ is linearly connected is somewhat limited, so we must move away from linearly connected considerations, and this includes the modified exponent $Q_X$. Consequently, we return to the original critical exponent $Q_N$ in order to study the conformal dimension. However, the UWS property is not well-suited to estimate $Q_N$. Fortunately, a straightforward discretization of it is.

\begin{definition}
A doubling metric space $(X,d)$ has the discrete UWS property if there is a constant $C \geq 1$ such that for any $x \in X$ and $0< r < C^{-1}\diam(X)$, for every $0<\e<1$, there is a finite set $K \subset B(x,2r)$, with $\#K \leq C$, such that every discrete $\e r$-path from $B(x,r)$ to $X \backslash \cl{B}(x,2r)$ intersects the $\e r$-neighborhood of $K$.
\end{definition}

Let us observe that the discrete UWS property, along with the associated constant $C$, are invariant under scaling the metric space $X$. We will use this fact later.

In working with the discrete UWS property, the following lemma is often helpful. It gives some flexibility in the constants that are used.

\begin{lemma} \label{weakuws}
Suppose that $X$ is doubling and there are constants $C,L \geq 2$ for which the following holds. For any $x \in X$ and $0<r < C^{-1}\diam(X)$, for every $0<\e<C^{-1}$, there is a finite set $K \subset B(x,Lr)$, with $\#K \leq C$, such that every discrete $\e r$-path from $B(x,r)$ to $X \backslash \cl{B}(x,Lr)$ intersects the $C \e r$-neighborhood of $K$. Then $X$ has the discrete UWS property, with constant depending only on $C$, $L$, and the doubling constant of $X$.
\end{lemma}

\begin{proof}
Fix $x \in X$, $0< r < C^{-1}\diam(X)$, and $0<\e<1$. If $\e \geq 1/2LC$ and $P$ is a maximal $\e r$-separated set in $B(x,2r)$, then $\#P \leq C_1$, where $C_1$ depends only on $C$, $L$, and the doubling constant of $X$. We may then take $K = P$ to verify the discrete UWS property.

Thus, we may assume that $0<\e<1/2LC$. Let $P$ be a maximal $r/LC$-separated set in $B(x,r)$, so that $\#P \leq C_2$, with $C_2$ depending only on $C$, $L$ and the doubling constant of $X$. By our assumption, for each $z \in P$, there is a set $K_z \subset B(z,r/2)$ with $\#K_z \leq C$ such that every discrete $\e r$-path from $B(z,r/2L)$ to $X \backslash \cl{B}(z,r/2)$ intersects the $C\e r$-neighborhood of $K_z$. Notice that we are using the assumption with parameters $r/2L$ and $2L\e < 1/C$. Now, for each $y \in K_z$, let $P_{z,y}$ be a maximal $\e r$-separated set in $B(y, C\e r) \subset B(x, 2r)$, so that $\#P_{z,y} \leq C_3$ with $C_3$ depending only on $C$ and the doubling constant of $X$. Finally, let
$$K = \bigcup_{z \in P} \bigcup_{y \in K_z} P_{z,y}$$
so that $\# K \leq CC_2 C_3$.

If $x_0, \ldots, x_\ell$ is a discrete $\e r$-path from $B(x,r)$ to $X \backslash \cl{B}(x,2r)$, then there is $z \in P$ for which $d(x_0,z) < r/LC$. Thus, $x_0,\ldots,x_\ell$ is a discrete $\e r$-path from $B(z,r/2L)$ to $X \backslash \cl{B}(z,r/2)$. In particular, there is $y \in K_z$ such that $d(x_i,y) < C \e r$ for some $1 \leq i \leq \ell$. Finally, this implies that $x_i$ lies in the $\e r$-neighborhood of $P_{z,y}$ and thus in the $\e r$-neighborhood of $K$.
\end{proof}

Shortly we will see that the discrete UWS property is closely connected to the Ahlfors-regular conformal dimension, so it makes sense to ask whether it is invariant under quasisymmetric maps. We should remark that by Lemma \ref{weakuws}, it is easy to see that the property is invariant under bi-Lipschitz maps between doubling metric spaces. Unfortunately, quasisymmetric invariance is not clear in general, but it is true when the spaces are doubling and linearly connected.

\begin{lemma}
Suppose that $X$ is doubling, linearly connected, and has the discrete UWS property. Then every metric space in $\mathcal{J}(X)$ has the discrete UWS property as well. In particular, every quasi-circle has the discrete UWS property.
\end{lemma}

\begin{proof}
Fix $(Y,d') \in \mathcal{J}(X)$ and let $f \colon Y \rightarrow X$ be an $\eta$-quasisymmetric homeomorphism. It is well-known (and easy to show) that $Y$ is doubling and linearly connected as well. Moreover, there is $L > 1$, depending only on the distortion function $\eta$, such that the following holds. For each $y \in Y$ and $r >0$, there is a radius $R >0$ for which
$$\begin{aligned}
B(f(y),R/L) \subset f(B(y,r/2)) &\subset f(B(y,r)) \\
&\subset B(f(y),R) \subset B(f(y),2R) \subset f(B(y,Lr)).
\end{aligned}$$
One can easily verify this by using the fact that $f^{-1}$ is also quasisymmetric, with distortion function depending only on $\eta$.

Fix $y \in Y$, $0<r<\diam(Y)/2$, and $0<\e<1/10$ small enough that $\delta := 1/L\eta(4L/\e) < 1$. Let $x = f(y)$ and let $R >0$ be the radius given by the previous paragraph. As $X$ has the discrete UWS property, there is a finite set $K' \subset B(x, 2R)$ with $\#K' \leq C$ such that every discrete $\delta R$-path in $X$ from $B(x,R)$ to $X \backslash \cl{B}(x,2R)$ meets the $\delta R$-neighborhood of $K'$. Define $K = f^{-1}(K') \subset B(y, Lr)$ so that $\#K \leq C$ as well.

Now, let $y_0,\ldots,y_\ell$ be a discrete $\e r$-path in $Y$ from $B(y,r)$ to $X \backslash \cl{B}(y,Lr)$. Passing to a sub-path, we may assume that $y_0,\ldots,y_\ell$ is a discrete $2\e r$-path with $d'(y_i,y_{i-1}) \geq \e r/2$ and, moreover, that $y_i \in B(y,2Lr) \backslash B(y,r/2)$ for each $i$. Let $x_i = f(y_i)$ so that $x_0,\ldots,x_{\ell}$ is a discrete path in $X$ from $B(x,R)$ to $X \backslash \cl{B}(x,2R)$ that lies outside of $B(x,R/L)$. Observe then that
$$\frac{d(x_i,x)}{d(x_i,x_{i-1})} \leq \eta \lp \frac{d'(y_i,y)}{d'(y_i,y_{i-1})} \rp \leq \eta \lp \frac{2Lr}{\e r/2} \rp = \eta(4L/\e),$$
and therefore
$$d(x_i,x_{i-1}) \geq \frac{d(x_i,x)}{\eta(4L/\e)} \geq \frac{R}{L\eta(4L/\e)} = \delta R.$$
for each $1 \leq i \leq \ell$.

Let $\lambda \geq 1$ be the constant of linear connectivity for $X$. Then for each $i$, there is a compact, connected set $E_i$ with $x_i,x_{i-1} \in E_i$ and $\diam(E_i) \leq \lambda d(x_i,x_{i-1})$. The set $E = \cup_{i=1}^\ell E_i$ is also compact and connected, and it joins $B(x,R)$ to $X \backslash \cl{B}(x,2R)$. Thus, there is a discrete $\delta R$-path consisting of points in $E$ that joins $B(x,R)$ to $X \backslash \cl{B}(x,2R)$. By the definition of $K'$, this discrete path must meet the $\delta R$-neighborhood of $K'$. In particular, there is $1 \leq i \leq \ell$ for which $E_i$ meets the $\delta R$-neighborhood of $K'$. Using that $\diam(E_i) \leq \lambda d(x_i,x_{i-1})$ and $d(x_i,x_{i-1}) \geq \delta R$, we can conclude that $K'$ intersects the ball $B(x_i, 2\lambda d(x_i,x_{i-1}))$. Let $z'$ be a point of this intersection, and let $z = f^{-1}(z') \in K$. We then have
$$\frac{1}{2\lambda} < \frac{d(x_i,x_{i-1})}{d(x_i,z')} \leq \eta \lp \frac{d'(y_i,y_{i-1})}{d'(y_i,z)} \rp \leq \eta \lp \frac{2\e r}{d'(y_i,z)} \rp,$$
so that $d'(y_i,z) < 2\e r/\eta^{-1}(1/2\lambda)$. Thus, the discrete path $y_0,\ldots,y_\ell$ intersects the $2\e r/\eta^{-1}(1/2\lambda)$-neighborhood of $K$. Using Lemma \ref{weakuws}, we conclude that $Y$ has the discrete UWS property.
\end{proof}

The following lemma shows that for doubling, linearly connected metric spaces, the discrete UWS property is more general than the UWS property.

\begin{lemma}
If $X$ is doubling, linearly connected, and has the UWS property, then $X$ has the discrete UWS property.
\end{lemma}

\begin{proof}
Suppose that $X$ has the UWS property with constant $C$. Fix $x \in X$ and $0<r<C^{-1} \diam(X)$. Let $K \subset \cl{B}(x,2r)$ be a finite set for which $\#K \leq C$ and no connected component of $X \backslash K$ intersects both $B(x,r)$ and $X \backslash \cl{B}(x,2r)$. Now let $0<\e<1/2$ and let $x_0,\ldots,x_{\ell}$ be a discrete $\e r$-path in $X$ that joins $B(x,r)$ to $X \backslash \cl{B}(x,2r)$. Without loss of generality, we may assume that each $x_i$ lies in $B(x,3r) \backslash \cl{B}(x,r/2)$. 

Let $\lambda$ be the constant of linear connectivity, so for each $1 \leq i \leq \ell$, there is a compact, connected set $E_i$ with $x_i, x_{i-1} \in E_i$ and $\diam(E_i) \leq \lambda \e r$. Let $E = \cup_{i=1}^\ell E_i$ so that $E$ is compact, connected, and intersects both $B(x,r)$ and $X \backslash \cl{B}(x,2r)$. In particular, $E$ must intersect $K$, so there is $1\leq i\leq \ell$ with $E_i \cap K \neq \emptyset$. As $E_i \subset B(x_i, 2\lambda \e r)$, we see that $x_i$ lies in the $2\lambda \e r$-neighborhood of $K$. Appealing to Lemma \ref{weakuws}, we conclude that $X$ has the discrete UWS property.
\end{proof}

We now prove Theorem \ref{cdim}, one of the two results needed to obtain the main theorem. In light of the previous lemma, we can think of it as an extension of \cite[Theorem 1.2]{CP14}. The proofs are almost identical, though.

\begin{thmcdim}
If $X$ is compact, doubling, uniformly perfect, and has the discrete UWS property, then $\Cdim_{\AR}(X) \in \{0,1\}$. It is equal to $0$ if and only if $X$ is uniformly disconnected.
\end{thmcdim}

\begin{proof}
We may assume that $X$ has at least two points. First we establish that $\Cdim_{\AR}(X) \leq 1$. Let $C$ be the constant from the discrete UWS property. Let $G_k$ denote the graphs constructed from $X$ using parameters $a > \max\{10, C/\diam(X) \}$ and $\lambda \geq 32$, so that the vertex sets $P_k$ are maximal $a^{-k}$-separated sets in $X$, and two vertices $x,y \in P_k$ are joined by an edge if
$$B(x,\lambda a^{-k}) \cap B(y,\lambda a^{-k}) \neq \emptyset.$$
Fix $m \in \N$ and let $x \in P_m$. Finally, fix $p >1$. Our goal is to estimate the quantity $\mod_p(\Gamma_k(x), G_{m+k})$ for large values of $k$.

For simplicity, let $r = a^{-m} < C^{-1}\diam(X)$ and let $\e = 2\lambda a^{-k} <1/10$, where $k$ is very large. Let $K \subset B(x, 2r)$ be a finite set coming from the discrete UWS property, so that $\#K \leq C$ and every discrete $\e r$-path from $B(x,r)$ to $X \backslash \cl{B}(x,2r)$ intersects the $\e r$-neighborhood of $K$. Then, let
$$\tilde{K} = \{ z \in P_{m+k} : B(z,\e r) \cap K \neq \emptyset \}.$$
Notice that $\# \tilde{K} \leq \#K \cdot C' \leq CC'$, where $C'$ depends only on the doubling constant of $X$, on $\lambda$, and on $a$. Moreover, every vertex path in $G_{m+k}$ from $P_{m+k} \cap B(x,r)$ to $P_{m+k} \cap (X \backslash \cl{B}(x,2r))$ must include some vertex in $\tilde{K}$. Indeed, the vertex path in $G_{m+k}$ forms a discrete $\e r$-path in $X$ from $B(x,r)$ to $X \backslash \cl{B}(x,2r)$ and therefore must intersect the $\e r$-neighborhood of $K$.

For each $z \in \tilde{K}$, let $\Gamma_z$ be the collection of vertex paths in $G_{m+k}$ that include $z$ and include a vertex that lies in $X \backslash \cl{B}(z,r/2)$. It is then clear that $\Gamma_k(x) \subset \bigcup_{z \in \tilde{K}} \Gamma_z$, so by sub-additivity of modulus,
$$\mod_p(\Gamma_k(x), G_{m+k}) \leq \sum_{z \in \tilde{K}} \mod_p(\Gamma_z, G_{m+k}).$$
Thus, we need to bound $\mod_p(\Gamma_z, G_{m+k})$ for each $z \in \tilde{K}$. 

To this end, fix $z \in \tilde{K}$, and for each $1 \leq i \leq k/2$, consider the balls $B(z,a^{-i}r)$ and $B(z,2a^{-i}r)$. By the discrete UWS property, there is a finite set $K_{z,i} \subset B(z,2a^{-i}r)$ with $\#K_{z,i} \leq C$ such that every discrete $\e r$-path from $B(z,a^{-i}r)$ to $X \backslash \cl{B}(z,2a^{-i}r)$ intersects the $\e r$-neighborhood of $K_{z,i}$. Here, we are using that $\e r = 2\lambda a^{-k} r$ is much smaller than the radius $a^{-i}r$ because $i \leq k/2$. Also, note that we may assume $K_{z,i} \cap B(z, 5a^{-i-1}r) = \emptyset$, as $a>10$. Now, let 
$$\tilde{K}_{z,i} = \{ v \in P_{m+k} : B(v,\e r) \cap K_{z,i} \neq \emptyset \},$$
be the corresponding set of vertices in $P_{m+k}$, and note that $\# \tilde{K}_{z,i} \leq \# K_{z,i} \cdot C' \leq CC'$. These sets are disjoint for $1 \leq i \leq k/2$, again because $a >10$ and $\e r$ is much smaller than $a^{-i}r$. Moreover, each vertex path in $\Gamma_z$ intersects both $B(z,a^{-i}r)$ and $X \backslash \cl{B}(z,2a^{-i}r)$, so a sub-path of it forms a discrete $\e r$-path from $B(z,a^{-i}r)$ to $X \backslash \cl{B}(z,2a^{-i}r)$. In particular, it must include a vertex in $\tilde{K}_{z,i}$ for each $i$.

Define a weight function $\rho \colon P_{m+k} \rightarrow \R$ by 
$$\rho(v) =
\begin{cases}
3/k & \text{ if } v \in \bigcup_{i=1}^{k/2} \tilde{K}_{z,i} \\
0 & \text{ otherwise.}
\end{cases}$$
Every vertex path $\gamma \in \Gamma_z$ contains a vertex in $\tilde{K}_{z,i}$ for each $1\leq i \leq k/2$, so we have $\sum_{v \in \gamma} \rho(v) \geq 1$. Thus, $\rho$ is admissible for $\Gamma_z$, and we can estimate
$$\mod_p(\Gamma_z, G_{m+k}) \leq \sum_{v \in P_{m+k}} \rho(v)^p  = \# \lp \bigcup_{i=1}^{k/2} \tilde{K}_{z,i} \rp 3^p k^{-p}  \leq 3^pC C'  \cdot k^{1-p}.$$
This gives the bound
$$\mod_p(\Gamma_k(x), G_{m+k}) \leq \# \tilde{K} \cdot 3^p C C' \cdot k^{1-p} \leq 3^p(CC')^2 \cdot k^{1-p}.$$
As $m \in \N$ and $x \in P_m$ were arbitrary, we obtain $M_p(k) \leq 3^p(CC')^2 \cdot k^{1-p}$, which tends to zero as $k \rightarrow \infty$. Moreover, $p>1$ was arbitrary, so we can conclude that $\Cdim_{\AR}(X) = Q_N \leq 1$.

We now observe that $Q_N$ cannot assume values in $(0,1)$. Indeed, if $X$ is uniformly disconnected then for each $x \in P_m$, the collection $\Gamma_k(x)$ is empty for sufficiently large $k$ and so $\liminf_{k \rightarrow \infty} \mod_p(\Gamma_k(x),G_{m+k}) = 0$. Thus, $M_p =0$ for each $p$, and we have $Q_N=0$. On the other hand, if $X$ is not uniformly disconnected, then for each $k \in \N$, there is some $m \in \N$ and $x \in P_m$ for which $\Gamma_k(x)$ is non-empty. When $0<p<1$, this ensures that $\mod_p(\Gamma_k(x),G_{m+k}) \geq 1$, so we see that $M_p(k) \geq 1$. In particular, $M_p \geq 1$ for all $0<p<1$, so we must have $Q_N \geq 1$.
\end{proof}

In order to prove Theorem \ref{John1}, it remains to verify Theorem \ref{Johnuws}, which says that boundaries of John domains have the discrete UWS property. Thus, we now turn our attention back to the planar setting.

\section{Conformal dimension of boundaries of John domains} \label{Johnsec}

As we are working with essentially discrete concepts, it is not surprising that we will need some input from combinatorics. Our proof of Theorem \ref{Johnuws}, which is really the heart of this paper, will hinge on the following fact. It provides a dual relationship between separation and connection in finite graphs. Here, we say that a set of vertices $S$ in a graph separates two other sets of vertices, $A$ and $A'$, if every vertex path from $A$ to $A'$ intersects $S$ nontrivially.

\begin{menger}[{\cite[Theorem 3.3.1]{Die}}]
Let $G = (V,E)$ be a finite graph, and let $A,A' \subset V$ be non-empty and disjoint. Then the minimal size of a vertex set that separates $A$ and $A'$ in $G$ is equal to the maximal number of pairwise disjoint vertex paths that connect $A$ and $A'$.
\end{menger}

Let us recall Theorem \ref{Johnuws} before giving its proof, which relies heavily on Menger's theorem. We also remind the reader that all distances are taken in the spherical metric, though we will use the notation $|x-y|$.

\begin{thmJohnuws}
If $\Omega \subset \hat{\C}$ is an $L$-John domain, then $\partial \Omega$ has the discrete UWS property, with constant depending only on $L$.
\end{thmJohnuws}

\begin{proof}
Let $z_0$ be the base-point of $\Omega$ so that $\Omega$ is an $L$-John domain with respect to $z_0$. As the $L$-John condition and the discrete UWS condition (along with its constant) are both invariant under scaling, we may assume that $\delta_{\Omega}(z_0) = 1$.

Fix $x \in \partial \Omega$, $0<r< 1/10$, and $0<\e<1/100L$. We must produce a finite set $K$ which verifies the discrete UWS property for these parameters. To this end, let $P$ be a maximal $2 \e r$-separated set in $B(x,3r) \cap \partial \Omega$, and let $G$ be the graph with vertex set $P$, where $u,v \in P$ are joined by an edge if 
$$B(u,5\e r) \cap B(v,5\e r) \neq \emptyset.$$ 
Throughout this proof, balls will refer to spherical balls in $\hat{\C}$, not just to their intersections with $\partial \Omega$. Thus, the rule used to join $u,v \in P$ by an edge in $G$ requires only that the balls $B(u,5\e r)$ and $B(v,5\e r)$ intersect in $\hat{\C}$. Observe also that $\{B(v, \e r) : v \in P\}$ is a disjoint collection of balls.

Define the sets 
$$A = \{v \in P : B(v,2\e r) \cap B(x,r) \neq \emptyset \}$$
$$A' = \{v \in P : B(v,2\e r) \cap (\hat{\C} \backslash \cl{B}(x,2r)) \neq \emptyset \},$$
so that $A$ and $A'$ are disjoint vertex sets in $G$. Our eventual goal is to show that there is a uniformly bounded number of disjoint vertex paths in $G$ that connect $A$ to $A'$, where the bound depends only on $L$.

First, let us see why this gives the desired conclusion. By Menger's theorem, applied to the graph $G$ with vertex sets $A$ and $A'$, we can find a set $K \subset P$, with $\#K$ uniformly bounded depending only on $L$, that separates $A$ from $A'$ in $G$. Now, suppose that $x_0,\ldots,x_\ell \in \partial \Omega$ is a discrete $\e r$-path from $B(x,r)$ to $\hat{\C} \backslash \cl{B}(x,2r)$, so that $x_0 \in B(x,r)$ and $x_\ell \in \hat{\C} \backslash \cl{B}(x,2r)$. Without loss of generality, we may assume that $x_i \in B(x,3r)$ as well. For each $x_i$ there is $v_i \in P$ with $|x_i - v_i| < 2\e r$. This means that $|v_i - v_{i-1}| < 5\e r$, so $v_0,\ldots, v_\ell$ is a vertex path in $G$ with $v_0 \in A$ and $v_\ell \in A'$. In particular, there is some $v_i \in K$, which gives $\dist(x_i,K) < 2\e r$. As $\partial \Omega$ is doubling, we may appeal to Lemma \ref{weakuws} to conclude that $\partial \Omega$ has the discrete UWS property.

It therefore suffices to bound the number of disjoint vertex paths in $G$ that connect $A$ to $A'$. To this end, let us fix $\eta_1,\ldots,\eta_m$ to be such vertex paths, which we express as
$$\eta_i = x_i(0),x_i(1),\ldots,x_i(\ell_i),$$
where $x_i(j) \in P$ for each $i$ and $j$. Without loss of generality, we may assume that $\eta_i$ is minimal, in the sense that no proper subset of its vertices forms a path in $G$ that connects $A$ to $A'$. In particular, this means that the collection of vertices $x_i(j)$ are all distinct, so the balls $B(x_i(j), \e r)$ are all disjoint. Now, for each $x_i(j)$, we choose a point $y_i(j) \in B(x_i(j), \e r/2)$ so that the following two conditions hold:
\begin{enumerate}
\item[(i)] the spherical segment $[y_i(j-1),y_i(j)]$ lies in $B(x_i(j-1),5\e r) \cup B(x_i(j),5\e r)$ for each $1 \leq j \leq \ell_i$, and
\item[(ii)] the collection $\{y_i(j)\}_{i,j}$ is in general position: no three points lie on a common line, and no three lines determined by this collection intersect at a common point.
\end{enumerate}
Note that the balls $B(y_i(j), \e r/2)$ are all disjoint as well.

Now, fix $1 \leq i \leq m$, and let $\alpha_i$ be the piecewise-linear path consisting of the segments $[y_i(j-1),y_i(j)]$ for $1 \leq j \leq \ell_i$, parameterized according to this order. A priori, these paths might not be simple, and they very well might intersect each other. However, the fact that $\{y_i(j)\}_{i,j}$ are in general position implies that there are only finitely-many intersections, they are all transverse, and none of the $y_i(j)$ are intersection points. In fact, every intersection is simply the intersection of two open spherical segments $(y_i(j-1),y_i(j))$ and $(y_{i'}(j'-1),y_{i'}(j'))$.

Let us modify the paths $\alpha_i$ to obtain disjoint arcs as follows. At each intersection point $q$ consider a very small ball $B(q,\delta)$, where $\delta >0$ is chosen so that
\begin{enumerate}
\item[(i)] $\delta$ is much less than the distance between $q$ and any of the points $y_i(j)$,
\item[(ii)] $\delta$ is much less than the distance between $q$ and any other intersection point, and
\item[(iii)] $B(q,\delta)$ still lies in one of the balls $B(x_i(j), 5\e r)$, eg. the one that $q$ lies in.
\end{enumerate}
Let $\alpha_s$ and $\alpha_t$ be the paths that intersect at $q$ (here $s=t$ is allowed), so their intersection with $B(q,\delta)$ forms a star at $q$ with four endpoints on $\partial B(q,\delta)$. According to the parameterizations of $\alpha_s$ and $\alpha_t$, two of the four prongs of this star are directed inward toward $q$ and the other two are directed outward from $q$. We may therefore reconnect the four segments inside $B(q,\delta/2)$ to produce two arcs that are disjoint in $B(q,\delta)$ and which enter and exit the ball in the same directions and at the same points as did the original four prongs (see Figure \ref{prongs}).

\begin{figure}[h] 
  \centering
    \includegraphics[width=0.75\textwidth]{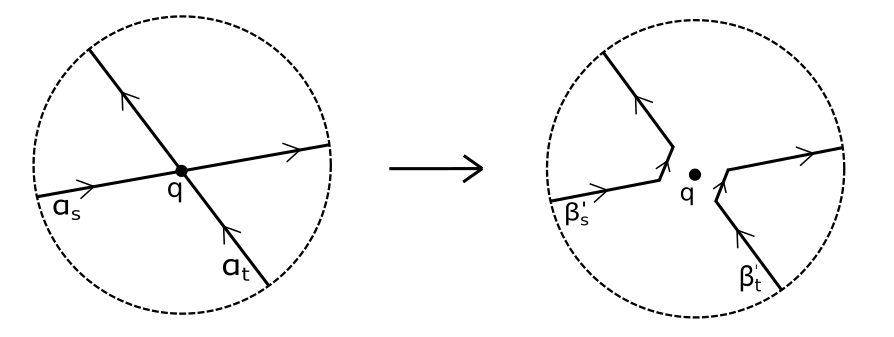}
    \caption{Reconnecting the four prongs to obtain disjoint arcs in $B(q,\delta)$}
\label{prongs}
\end{figure}

We perform this process at all intersection points of the paths $\alpha_1,\ldots,\alpha_m$, thereby producing a collection of disjoint arcs $\beta'_1,\ldots,\beta'_m$ (and possibly some disjoint loops, which we discard). This process does not alter $\alpha_i$ near its initial point or its endpoint, so the family $\beta'_1,\ldots,\beta'_m$ has the same collection of initial points and endpoints as did $\alpha_1,\ldots,\alpha_m$. In particular, each $\beta'_i$ is an arc that connects $B(x,7r/6)$ to $\hat{\C} \backslash \cl{B}(x,11r/6)$, and $\beta'_i$ is contained in the $5 \e r$-neighborhood of $P \subset \partial \Omega$. Now let $\beta_i$ be a minimal sub-arc of $\beta'_i$ that has its initial point on $ \partial B(x,7r/6)$ and its endpoint on $\partial B(x,11r/6)$.

Recall that our goal is to show that $m$ is uniformly bounded. This is where we will use the John condition for $\Omega$. First, it will be helpful to re-order the arcs $\beta_1,\ldots,\beta_m$ to coincide with, say, the counter-clockwise orientation on their initial points, which all lie on the circle $ \partial B(x,7r/6)$. This means that for each $i$, the arc $\beta_i$ is contained in a single connected component of the twice-slit annulus
$$\lp \cl{B}(x,11r/6) \backslash B(x,7r/6) \rp \backslash (\beta_{i-1} \cup \beta_{i+1})$$
and this component contains no other $\beta_j$. Here, $i-1$ and $i+1$ should be interpreted modulo $m$.

We now find, for each $i$, a John arc that ends near $\beta_i$. More precisely, for each arc $\beta_i$, let $u_i$ be a point on $\beta_i \cap \partial B(x,3r/2)$. As $\beta_i$ is a sub-arc of $\beta_i'$, there must be a point $v_i \in P \subset \partial \Omega$ with $|u_i -v_i| < 5\e r$. There are points in $\Omega$ arbitrarily close to $v_i$, so we can find $z_i \in \Omega$ with $|u_i - z_i| < 5\e r$. By the John assumption, there is an $L$-John arc in $\Omega$ from $z_0$ to $z_i$. Notice that $2r < 1 = \delta_{\Omega}(z_0)$, so we know that $z_0$ lies outside of $\cl{B}(x,2r)$. Thus, there is a sub-arc of this John arc that lies entirely in the annulus $\cl{B}(x,11r/6) \backslash B(x,7r/6)$ and still has tip $z_i$. Call this sub-arc $\gamma_i$. It is possible that $\gamma_i$ may enter the annulus from either complementary component. It must, however, pass through at least one of $\partial B(x,4r/3)$ and $\partial B(x,5r/3)$. 

Let $w_i$ be a point at which $\gamma_i$ intersects $\partial B(x,4r/3)$ or $\partial B(x,5r/3)$, and let $D_i = B(w_i, r/12L)$. Note that 
$$\diam(\gamma_i[w_i,z_i]) \geq |w_i-z_i| \geq |u_i - w_i| - |u_i - z_i| \geq r/6 - 5\e r > r/12$$
so that $D_i \subset \Omega$. A similar argument shows that $D_i$ does not intersect any of the arcs $\beta_1,\ldots,\beta_m$, which are all in the $5\e r$-neighborhood of $\partial \Omega$. Now, If the disks $D_i$ were disjoint, then we would be finished by volume considerations, as they all lie in the ball $B(x,2r)$. This may fail, though, because the John arcs $\gamma_i$ might cross some of the arcs $\beta_1,\ldots,\beta_m$. We claim, however, that a definite proportion of the disks $D_i$ are disjoint.

For each $i$, consider the set of indices $J_i = \{ j : \dist(\gamma_i, \beta_j) \leq 5\e r \}$, which is non-empty because $i \in J_i$. Fix $j \in J_i$, so there is $u \in \beta_j$ and $z \in \gamma_i$ with $|u-z| \leq 5\e r$. This means that $\delta_{\Omega}(z) \leq 10\e r$, and the John condition then implies that $|z-z_i| \leq 10L\e r$. Moreover, as $u \in \beta_j \subset \beta'_j$, there is a point $y_j \in \beta_j'$, that is in the set $\{y_{i'}(j')\}_{i',j'}$ of endpoints considered earlier, such that $|u-y_j| < 10\e r$. Notice that
$$|y_j - z_i| \leq |y_j - u| + |u-z| + |z-z_i| < 10\e r+ 5\e r +10L\e r \leq 25L\e r.$$
In this way, we obtain, for each $j \in J_i$, a point $y_j \in B(z_i, 25L\e r) \cap \beta_j'$, and these points are distinct because the arcs $\beta_j'$ are disjoint. Recall, though, that all of the balls $B(y_{i'}(j'),\e r/2)$ were disjoint by the way we chose the points $y_{i'}(j')$. In particular, the collection $\{B(y_j,\e r/2)\}_{j \in J_i}$ is disjoint and lies in $B(z_i, 26L\e r)$. Volume considerations then imply that $\#J_i \leq C$, where $C$ depends only on $L$.

We may now choose $I  = \{i_1,\ldots, i_k\} \subset \{1,\ldots,m\}$ to be a set of indices that contains at most one element from each set $J_i$, such that $k \geq m/2C$. It suffices to show that $k$ is uniformly bounded by a constant depending only on $L$. Consider the set of crossing arcs $\beta_{i_1},\ldots,\beta_{i_k}$ and their corresponding John arcs $\gamma_{i_1}, \ldots, \gamma_{i_k}$. For $j,j'$ distinct, we have $\dist(\gamma_{i_j}, \beta_{i_{j'}}) > 5 \e r$; otherwise both $i_j$ and $i_{j'}$ would be in $J_{i_j}$. In particular, $\gamma_{i_j}$ lies entirely in one connected component of
$$\lp \cl{B}(x,11r/6) \backslash B(x,7r/6)\rp \backslash (\beta_{i_{j-1}} \cup \beta_{i_{j+1}}),$$
i.e., the component that $\beta_{i_j}$ lies in. Moreover, the disk $D_{i_j}$ lies entirely in this same component. Indeed, $D_{i_j}$ does not intersect $\beta_{i_{j-1}}$ or $\beta_{i_{j+1}}$, and $\gamma_{i_j}$ connects the center of $D_{i_j}$ to the point $z_{i_j} \in \Omega$, which has $\dist(z_{i_j}, \beta_{i_j}) < 5\e r$ (the segment from $z_{i_j}$ to the nearest point on $\beta_{i_j}$ cannot meet $\beta_{i_{j-1}}$ or $\beta_{i_{j+1}}$).

Hence, we can conclude that if $D_{i_j}$ and $D_{i_{j'}}$ are not disjoint, then $j' = j \pm 1 \mod k$. Consequently, at least $k/3$ of the disks $D_{i_1}, \ldots, D_{i_k}$ are disjoint. All of these lie in the ball $B(x,2r)$, so again volume considerations guarantee that $k$ is bounded by a uniform constant depending only on $L$.
\end{proof}

Combining this with Theorem \ref{cdim}, we obtain Theorem \ref{John1} immediately. In particular, if $\Omega$ is a John domain such that $\partial \Omega$ is connected and has at least two points, then $\Cdim(\partial \Omega) = \Cdim_{\AR}(\partial \Omega)=1$. We have already mentioned that the conformal dimension need not be attained by a metric space in $\mathcal{J}(\partial \Omega)$ or in $\mathcal{J}_{\AR}(\partial \Omega)$. Instead, similar to the questions raised in Section \ref{qtree}, it would be interesting to know whether the conformal dimension of $\partial \Omega$ can be calculated using planar sets.

\begin{question}
Can one characterize the Jordan curves $\Gamma$ for which $\Cdim_{\AR}(\Gamma) = 1$ and is attained? Is it attained by a $1$-regular curve in $\hat{\C}$ that is the image of $\Gamma$ under a quasiconformal homeomorphism of $\hat{\C}$? 
\end{question}

Related to these considerations, we should mention the following fact, which is not difficult to establish using Menger's theorem.

\begin{prop} \label{1reg}
If a metric space $X$ is Ahlfors 1-regular, then $X$ has the discrete UWS property.
\end{prop}

\begin{proof}
First, we note that Ahlfors regular metric spaces are always doubling. Fix $x \in X$, $0<r<\diam(X)/2$, and $0<\e<1/50$. Let $P$ be a maximal $2\e r$-separated set in $B(x,3r)$, so that $P$ is finite. Let $G$ be the graph whose vertex set is $P$ with an edge between $u,v \in P$ if 
$$B(u,5\e r) \cap B(v, 5\e r) \neq \emptyset.$$
As in the previous proof, define the sets 
$$A = \{v \in P : B(v,2\e r) \cap B(x,r) \neq \emptyset \}$$
$$A' = \{v \in P : B(v,2\e r) \cap (X \backslash \cl{B}(x,2r)) \neq \emptyset \},$$
so that $A$ and $A'$ are disjoint vertex sets in $G$. Note that every discrete $\e r$-path $x_0,\ldots,x_\ell$ from $B(x,r)$ to $X \backslash \cl{B}(x,2r)$ gives a vertex path $v_0,\ldots,v_\ell$ from $A$ to $A'$ in $G$ by choosing $v_i \in P$ with $d(v_i,x_i) < 2\e r$. By Lemma \ref{weakuws}, it suffices to show that there are a uniformly bounded number of vertices in $G$ that separate $A$ from $A'$. Using Menger's theorem, this is equivalent to showing that there is a uniformly bounded number of disjoint vertex paths in $G$ from $A$ to $A'$.

To this end, let $\eta_1,\ldots, \eta_m$ be disjoint vertex paths from $A$ to $A'$, which we write as 
$$\eta_i = x_i(0),x_i(1),\ldots,x_i(\ell_i),$$
where $x_i(j) \in P$ for each $i$ and $j$. Recall that $P \subset B(x,3r)$ by definition. For each $i$, we have
$$r/2 \leq d(x_i(0),x_i(\ell_i)) \leq \sum_{j=1}^{\ell_i} d(x_i(j-1),x_i(j)) \leq \ell_i  \cdot 10\e r,$$
so that $\ell_i \geq (20\e)^{-1}$. Moreover, the collection of balls 
$$\{B(x_i(j), \e r) : 1\leq i \leq m, \hspace{1pt} 0\leq j \leq \ell_i\},$$ 
which are all contained in $B(x, 4r)$, is disjoint because $P$ is $2\e r$-separated. Consequently, if $\mu$ is an Ahlfors 1-regular Borel measure on $X$, then we have
$$4Cr \geq \mu(B(x,4r)) \geq \sum_{i=1}^m \sum_{j=0}^{\ell_i} \mu(B(x_i(j), \e r)) \geq m \cdot \frac{1}{20\e} \cdot \frac{\e r}{C},$$
where $C$ is the constant from the Ahlfors regularity condition. Thus, $m \leq 80C^2$, which is a uniform bound, as desired.
\end{proof}

\section{Weak tangents of John domains} \label{WTsec}

The UWS property and the discrete UWS property are both defined by connectivity conditions that must hold at all locations and scales in the metric space. Namely, the condition must be true at each $x \in X$ and for all $0<r < C^{-1}\diam(X)$; we do not allow the scale $r$ to depend on the location $x$. Thus, these properties are stronger than any related infinitesimal condition, for example, where the scale is allowed to depend on $x$. 

In a similar spirit, one could investigate the infinitesimal geometry of metric spaces that have the discrete UWS property. More specifically, we will look at connectivity properties of tangent spaces at each point. To describe this, we need some definitions. First, a pointed metric space $(X,d,p)$ is a metric space $(X,d)$ along with a distinguished base-point $p \in X$.

\begin{definition}
A sequence of pointed metric spaces $(X_n, d_n, p_n)$ converges to a pointed metric space $(Y, d', p)$ in the Gromov-Hausdorff sense if for every $\e > 0$ and $R \geq1$, there is $N \in \N$ such that for $n \geq N$, there is a map $\phi_n \colon B(p_n,R) \rightarrow Y$ for which
\begin{itemize}
\item[\textup{(i)}] $\phi_n(p_n) = p$,
\item[\textup{(ii)}] $\phi_n(B(p_n,R))$ is $\e$-dense in $B(p,R)$, and
\item[\textup{(iii)}] $|d_n(x,y) - d'(\phi_n(x),\phi_n(y))| < \e$ for each pair $x,y \in B(p_n,R)$.
\end{itemize}
\end{definition}

We should note that there are several different, but equivalent, definitions of Gromov-Hausdorff convergence. The definition we use here comes from \cite{BK}. In what follows, we will use the fact that every sequence $(X_n, d_n, p_n)$ of uniformly doubling pointed metric spaces has a subsequence that converges in the Gromov-Hausdorff sense. Of particular interest to us are sequences that arise as re-scalings of a given metric space, either around a common point or around varying points. Informally, this corresponds to zooming into the metric space at certain locations, along scales tending to zero.

\begin{definition}
Let $(X,d)$ be a doubling metric space. We say that a complete pointed metric space $(Y,d', p)$ is a weak tangent of $X$ if there is a sequence of points $p_n \in X$ and scales $\lambda_n > 0$, with $\lim_{n \rightarrow \infty} \lambda_n = 0$, such that the pointed metric spaces $(X,\lambda_n^{-1} d, p_n)$ converge in the Gromov-Hausdorff sense to $(Y,d', p)$.
\end{definition}

If $X$ is doubling, then each re-scaling $(X,\lambda^{-1} d)$ is also doubling, with the same constant. Thus, doubling metric spaces always have weak tangents. Note that we require weak tangents to be complete. This is not a restrictive assumption: if $(Y,d',p)$ is a Gromov-Hausdorff limit of a sequence of pointed metric spaces, then the completion of $(Y,d',p)$ is also a Gromov-Hausdorff limit of the same sequence. In the literature, weak tangents coming from sequences with a fixed base-point $p_n=x \in X$ are often simply called tangents at $x$. Thus, every tangent is a weak tangent.

Before discussing metric spaces with the discrete UWS property in general, let us turn our attention to linearly connected spaces and, in particular, to quasi-circles. The following characterization is probably known, but it does not seem to be present in the literature. Here, by a metric circle, we simply mean a metric space that is homeomorphic to $\Sp^1$.

\begin{theorem} \label{metriccircle}
A doubling metric circle is a quasi-circle if and only if every weak tangent is connected.
\end{theorem}

A classical result of P. Tukia and J. V{\"a}is{\"a}l{\"a} \cite{TV} says that a metric circle is a quasi-circle if and only if it is doubling and linearly connected. Thus, it suffices to prove the following proposition.

\begin{prop} \label{qarc}
Let $X$ be a complete, connected, doubling, bounded metric space. Then $X$ is linearly connected if and only if every weak tangent of $X$ is connected.
\end{prop}

Our proof of this proposition, like many of our previous proofs, will use a discretization argument. The following lemma is central to this argument. For similar considerations, see \cite[Proposition 4 and Lemma 5]{BKplanes}.

\begin{lemma} \label{hops}
Suppose that $X$ is a complete metric space with the following property: there exists $L \geq 1$ such that any $x,y \in X$ can be joined by a discrete $d(x,y)/4$-path inside $B(x,L d(x,y))$. Then $X$ is $4L$-linearly connected.
\end{lemma}

\begin{proof}
We may suppose that $X$ has at least two points. Fix $x,y \in X$ distinct, and let $r = d(x,y)>0$. By assumption, there is a discrete $r/4$-path 
$$x=x(0), x(1), \ldots, x(n)=y$$ 
with $x(i) \in B(x(0),L r)$. Then for each $1\leq i \leq n-1$, there is a discrete $r/16$-path 
$$x(i) = x(i,0), x(i,1), \ldots, x(i,n_i) = x(i+1)$$ 
inside $B(x(i),L r/4)$. We continue building such paths inductively. Namely, if $x(i_1,\ldots,i_k)$ and $x(i_1,\ldots,i_k+1)$ are consecutive points in a discrete $r/4^k$-path built at step $k$, then our assumption guarantees that there is a discrete $r/4^{k+1}$-path
$$x(i_1,\ldots,i_k) = x(i_1,\ldots,i_k, 0), \ldots, x(i_1,\ldots,i_k, n_{i_1,\ldots,i_k}) = x(i_1,\ldots,i_k+1)$$
which lies in the ball $B(x(i_1,\ldots,i_k), L r/4^k)$. Without loss of generality, we can assume that $n_{i_1,\ldots,i_k} \geq 2$ for each tuple $(i_1,\ldots,i_k)$.

We now want to use these points to build a continuous map $f\colon [0,1] \rightarrow X$ connecting $x$ and $y$. We first define $f$ on a dense subset of $[0,1]$. To do this, let $a(i) = i/n$ for each $0\leq i \leq n$. If we have already defined $a(i_1,\ldots,i_k)$ and $a(i_1,\ldots,i_k+1)$, with $\ell = a(i_1,\ldots,i_k+1) - a(i_1,\ldots,i_k) >0$, then for each $0 \leq j \leq n_{i_1,\ldots,i_k}$, we let 
$$a(i_1,\ldots,i_k, j) = a(i_1,\ldots,i_k) + j \ell/n_{i_1,\ldots,i_k}.$$ 
Then $ \{ a(i_1,\ldots,i_k, j) : 0 \leq j \leq n_{i_1,\ldots,i_k} \}$ is an equally-spaced set in the interval $[a(i_1,\ldots,i_k), a(i_1,\ldots,i_k+1)]$, with beginning-point $a(i_1,\ldots,i_k, 0) = a(i_1,\ldots,i_k)$ and ending-point $a(i_1,\ldots,i_k, n_{i_1,\ldots,i_k}) = a(i_1,\ldots,i_k +1)$.

Define $f(a(i_1,\ldots,i_k)) = x(i_1,\ldots,i_k)$ for each choice of $i_1,\ldots,i_k$. This definition is compatible among the different values of $k$. If $E_k$ denotes the set of points of the form $a(i_1,\ldots,i_k)$ and $E= \cup_{k \in \N} E_k$, then it is clear that $E$ is dense in $[0,1]$, as each $n_{i_1,\ldots,i_k} \geq 2$. Also, note that $f(E) \subset B(x,2L r)$, as each $t \in E_k$ has
$$d(x,f(t)) \leq L r + L r/4 + \ldots + L r/4^k < 2L r.$$

Moreover, we claim that $f$ is uniformly continuous on $E$. Indeed, if $\e >0$, choose $k$ large enough so that $4L r/4^k < \e$, and then choose $\delta >0$ smaller than the least distance between consecutive points in $E_k$. If $|s-t| < \delta$ with $s,t \in E$, then there are three consecutive points $a_k,b_k,c_k \in E_k$ with $s,t \in [a_k,c_k]$. The construction of $f$ gives
$$d(f(s),f(b_k)), \hspace{2pt} d(f(t),f(b_k)) \leq L r/4^k + L r/4^{k+1} + L r/4^{k+2} + \ldots < 2L r/4^k,$$
so that $f(s),f(t) \in B(f(b_k),2L r/4^k)$. In particular, $d(f(s),f(t)) < 4L r/4^k < \e$.

As $X$ is complete, $f$ extends continuously to $[0,1]$. Then $f([0,1]) \subset \cl{B}(x,2L r)$ is a compact connected set joining $x$ and $y$. We conclude that $X$ is $4L$-linearly connected.
\end{proof}

\begin{proof}[\textbf{Proof of Proposition \ref{qarc}}]
Suppose that $(X,d)$ is linearly connected with constant $L \geq 1$. Let $(Y, d', p)$ be a weak tangent of $X$ so that it is the Gromov-Hausdorff limit of $(X,\lambda_n^{-1} d, p_n)$, where $\lambda_n \rightarrow 0$. Fix $u,v \in Y$, and let $R = 4L\max\{ d'(p,u), d'(p,v)\}$. Then let $\e = d'(u,v)/8$. We can find $n \in \N$ large and a map $\phi \colon B(p_n, \lambda_n R) \rightarrow Y$ such that the image $\phi(B(p_n, \lambda_n R))$ is $\e$-dense in $B(p, R)$ and 
$$|\lambda_n^{-1} d(x,y) - d'(\phi(x),\phi(y)) | < \e$$
for each pair of points $x,y \in B(p_n, \lambda_n R)$.

Choose $x,y \in B(p_n, \lambda_n R)$ for which $d'(\phi(x),u) < \e$ and $d'(\phi(y),v) < \e$. By linear connectivity, there is a compact connected set $E \subset X$ with $x,y \in E$ and $\diam(E) \leq Ld(x,y)$. We can then find a discrete $\lambda_n \e$-path $x= z_0, z_1,\ldots, z_{\ell} = y$ from $x$ to $y$ consisting entirely of points in $E$. Observe that
$$u, \phi(z_0), \phi(z_1),\ldots, \phi(z_\ell), v$$
is a discrete $2\e$-path from $u$ to $v$ in $Y$. Moreover, for each $i$ we have
$$d'(u,\phi(z_i)) \leq d'(\phi(x),\phi(z_i)) + \e \leq \lambda_n^{-1}d(x,z_i) + 2\e \leq L\lambda_n^{-1}d(x,y) + 2\e.$$
Note that we can bound $\lambda_n^{-1}d(x,y) \leq d'(\phi(x),\phi(y)) + \e \leq d'(u,v) + 3\e \leq 11\e$, so this gives 
$$d'(u,\phi(z_i)) \leq 11L\e + 2\e \leq 13L \e < 2Ld'(u,v).$$ 
Consequently, the sequence $u, \phi(z_0), \phi(z_1),\ldots, \phi(z_\ell), v$ is a discrete $d'(u,v)/4$-path inside the ball $B(u, 2Ld'(u,v))$. As $u,v \in Y$ were arbitrary, Lemma \ref{hops} implies that $Y$ is $8L$-linearly connected; in particular, it is connected.

For the converse, assume that every weak tangent of $X$ is connected, and suppose that $X$ were not linearly connected. By Lemma \ref{hops}, this means that for each $n \in \N$, there is a pair $x_n, y_n \in X$ that cannot be joined by a discrete $d(x_n,y_n)/4$-path inside $B(x_n, nd(x_n,y_n))$. We note that $d(x_n,y_n) \rightarrow 0$ as $n \rightarrow \infty$ because $X$ is connected and bounded. Now consider the sequence of pointed metric spaces $(X, \lambda_n^{-1} d, x_n)$ with $\lambda_n = d(x_n,y_n)$. As $X$ is doubling, this sequence is uniformly doubling, so there is a subsequence $n_k$ that converges in the Gromov-Hausdorff sense to a complete metric space $(Y,d',p)$.

Let $\e = 1/20$. By passing to a further subsequence, we can find maps 
$$\phi_{n_k} \colon B(x_{n_k}, \lambda_{n_k} k) \rightarrow Y$$
with $\phi_{n_k}(x_{n_k}) = p$, for which $\phi_{n_k}(B(x_{n_k}, \lambda_{n_k} k))$ is $\e$-dense in $B(p,k)$, and 
$$|\lambda_{n_k}^{-1}d(x,y) - d'(\phi_{n_k}(x),\phi_{n_k}(y))| < \e$$ 
for each pair $x,y \in B(x_{n_k}, \lambda_{n_k} k)$. Note that $\phi_{n_k}(y_{n_k}) \in B(p, 2)$ for each $k$, and $Y$ is proper (it is complete and doubling). By passing to yet another subsequence, we may assume that $\phi_{n_k}(y_{n_k})$ converges to a point $y \in Y$.

As $Y$ is connected, there is a discrete $\e$-path $p = u_0,u_1,\ldots,u_{\ell} = y$ in $Y$. Let $R = \max_i d'(p,u_i)$ and fix $k > 2R$ large enough that $d'(\phi_{n_k}(y_{n_k}), y) < \e$. As the image $\phi_{n_k}(B(x_{n_k}, \lambda_{n_k} k))$ is $\e$-dense in $B(p,R)$, for each $i$ we can find $v_i \in B(x_{n_k},\lambda_{n_k} k)$ with $d'(\phi_{n_k}(v_i), u_i) < \e$. In fact, we may take $v_0 = x_{n_k}$ and $v_\ell = y_{n_k}$. Note that $v_i \in B(x_{n_k}, 2R \lambda_{n_k})$ for each $i$ because
$$\lambda_{n_k}^{-1} d(x_{n_k},v_i) \leq d'(p, \phi_{n_k}(v_i)) + \e < d'(p, u_i) + 2\e < 2R.$$
Moreover, the sequence $x_{n_k} = v_0, v_1, \ldots, v_\ell = y_{n_k}$ is a discrete $4\e \lambda_{n_k}$-path. Indeed, for each $i$, we have
$$\lambda_{n_k}^{-1} d(v_i,v_{i+1}) \leq d'(\phi_{n_k}(v_i), \phi_{n_k}(v_{i+1})) + \e \leq d'(u_i,u_{i+1}) + 3\e \leq 4\e.$$
As $\e =1/20$ and $\lambda_{n_k} = d(x_{n_k},y_{n_k})$, we have shown that $x_{n_k}$ and $y_{n_k}$ can be connected by a discrete $d(x_{n_k},y_{n_k})/4$-path in the ball $B(x_{n_k}, 2R d(x_{n_k},y_{n_k}))$. Noting that $n_k \geq k > 2R$, we see that this contradicts our choice of the pair $x_{n_k}, y_{n_k}$.
\end{proof}

Let us now turn back to the discrete UWS property, boundaries of John domains, and their relationships to weak tangents. Our goal for the remainder of the section is to prove the following theorem, which is a partial analog of Theorem \ref{metriccircle}.

\begin{theorem} \label{WTjohn}
If $\Omega \subset \hat{\C}$ is an $L$-John domain with $\partial \Omega$ connected, then every weak tangent of $\partial \Omega$ has at most $N$ connected components, where $N < \infty$ depends only on $L$.
\end{theorem}

By Theorem \ref{Johnuws}, it suffices to prove the following proposition, which we mentioned in Section \ref{intro}. Note that the doubling constant of $\partial \Omega$ is uniform because $\partial \Omega$ is planar.

\begin{propuwsblowup}
Let $X$ be a complete, connected, doubling metric space that has the discrete UWS property with constant $C$. Then every weak tangent of $X$ has at most $N$ connected components, where $N$ depends only on $C$ and the doubling constant of $X$.
\end{propuwsblowup}

Before proving the proposition, it will be helpful to establish a lemma.

\begin{lemma} \label{weaktanuws}
If $X$ is doubling and has the discrete UWS property with constant $C$, then every weak tangent of $X$ also has the discrete UWS property with constant depending only on $C$ and the doubling constant of $X$.
\end{lemma}

\begin{proof}
Let $(Y,d', p)$ be a weak tangent of $X$ associated to the sequence $(X,\lambda_n^{-1}d, p_n)$, with $\lambda_n \rightarrow 0$. Fix $y \in Y$, $r >0$, and $0<\e<1/10$. Let $R = d'(p,y) + 6r$ and take $n \in \N$ large enough that there is $\phi_n \colon B(p_n,\lambda_n R) \rightarrow Y$ for which $\phi_n(p_n) = p$, the image $\phi_n(B(p_n,\lambda_n R))$ is $\e r$-dense in $B(p,R)$, and 
$$|\lambda_n^{-1}d(u,v) - d'(\phi_n(u),\phi_n(v))| < \e r$$ 
for each pair $u,v \in B(p_n,\lambda_n R)$. We may also assume that $\lambda_n \leq 1$.

Let $x \in B(p_n, \lambda_n R)$ be a point for which $d'(\phi_n(x),y) < \e r$. By the discrete UWS property in $X$, we can find a set $K' \subset B(x, 4 \lambda_n r)$ with $\# K' \leq C$ for which every discrete $4\e \lambda_n r$-path from $B(x,2\lambda_n r)$ to $X \backslash \cl{B}(x, 4\lambda_n r)$ meets the $4\e \lambda_n r$-neighborhood of $K'$. Let $K = \phi_n(K') \subset B(y,5r)$; note that this containment holds because
$$d'(\phi_n(z),y) \leq d'(\phi_n(z),\phi_n(x)) + \e r \leq \lambda_n^{-1} d(z,x) + 2\e r \leq 4r + 2\e r < 5r$$
for each $z \in K'$. It is also clear that $\# K \leq \#K' \leq C$.

Now let $y_0, \ldots, y_\ell$ be a discrete $\e r$-path in $Y$ that joins $B(y,r)$ and $Y \backslash \cl{B}(y,5r)$. Without loss of generality, we may assume that this discrete path lies in $B(y, 6r)$ and so is contained in $B(p,R)$. Consequently, there are points $x_i \in B(p_n, \lambda_n R)$ with $d'(\phi_n(x_i),y_i) < \e r$. Note that
$$\lambda_n^{-1} d(x_i,x_{i-1}) \leq d'(\phi_n(x_i),\phi_n(x_{i-1})) +\e r \leq d'(y_i,y_{i-1}) + 3\e r < 4 \e r,$$
so that $x_0,\ldots,x_\ell$ is a discrete $4\e \lambda_n r$-path. Moreover,
$$\lambda_n^{-1} d(x_0,x) \leq d'(\phi_n(x_0),\phi_n(x)) +\e r \leq d'(y_0,y) + 3\e r < 2r$$
and
$$\lambda_n^{-1} d(x_\ell,x) \geq d'(\phi_n(x_\ell),\phi_n(x)) - \e r \geq d'(y_\ell, y) - 3\e r > 4r,$$
so this discrete path joins $B(x,2\lambda_n r)$ to $X \backslash \cl{B}(x, 4\lambda_n r)$. In particular, there is $z \in K'$ with $d(z,x_i) < 4\e \lambda_n r$ for some $1\leq i \leq \ell$. Consequently,
$$d'(\phi_n(z), y_i) \leq d'(\phi_n(z), \phi_n(x_i)) + \e r \leq \lambda_n^{-1} d(z,x_i) + 2 \e r < 6 \e r$$
so $y_i$ lies in the $6\e r$-neighborhood of $K'$. Using Lemma \ref{weakuws}, we conclude that $Y$ has the discrete UWS property.
\end{proof}

In fact, our argument in the previous lemma shows that if $(X_n,d_n,p_n)$ are pointed metric spaces that are uniformly doubling and have the discrete UWS property with a uniform constant, then any Gromov-Hausdorff limit of this sequence also has the discrete UWS property. We will not need this more general statement, though.

\begin{proof}[\textbf{Proof of Proposition \ref{uwsblowup}}]
We may assume that $X$ has at least two points. Suppose that $(X, \lambda_n^{-1} d, p_n)$ converges in the Gromov-Hausdorff sense to $(Y,d',p)$, where $\lambda_n \rightarrow 0$. Let $C_1,\ldots, C_k$ be connected components in $Y$. Our goal is to give a uniform bound on $k$, where the bound depends only on the constant $C$ from the discrete UWS property and on the doubling constant for $X$.

To this end, we first note that each $C_i$ is unbounded. Indeed, fix $y \in C_i$ and fix $R \geq \max\{100, 4d'(p,y)\}$ large. Now fix $0 < \e <1$. Take $n \in \N$ large enough that $4\lambda_n < \diam(X)$ and let $\phi_n \colon B(p_n, 5\lambda_n R) \rightarrow Y$ be a map for which $\phi_n(p_n) = p$, the image $\phi_n(B(p_n, 5\lambda_n R))$ is $\e$-dense in $B(p, 5R)$, and
$$|\lambda_n^{-1}d(u,v) - d'(\phi_n(u),\phi_n(v))| < \e$$ 
for each pair $u,v \in B(p_n, 5 \lambda_n R)$. Let $x_n \in B(p_n,\lambda_n R)$ be a point for which $d'(\phi_n(x_n),y) < \e$. As $X$ is connected, we can find a discrete $\lambda_n \e$-path, beginning with $x_n$ and ending with a point $z_n \in B(x_n, 4\lambda_n R) \backslash \cl{B}(x_n, 2\lambda_n R)$. Note that $B(x_n, 4\lambda_n R) \subset B(p_n, 5\lambda_n R)$, so the image of this discrete path under $\phi_n$ forms a discrete $2\e$-path beginning with $\phi_n(x_n) \in B(y, \e)$ and ending with $\phi_n(z_n) \in Y \backslash \cl{B}(y, R)$.

We have therefore shown that for all $\e >0$, there is a discrete $\e$-path in $Y$ from $y$ to $Y \backslash \cl{B}(y, R)$. As $Y$ is doubling and complete, it is proper, and so there is a point $y' \in B(y, 2R) \backslash \cl{B}(y, R)$ such that for all $\e >0$, there is a discrete $\e$-path from $y$ to $y'$. Consequently, $y' \in C_i$ as well, so that $\diam(C_i) \geq R$. Because $R$ was arbitrary, we conclude that $C_i$ is unbounded.

Now, as each $C_1,\ldots,C_k$ is unbounded, we can find a radius $R$ large enough that each $C_i$ intersects $B(p,R/2)$ and $Y \backslash \cl{B}(p,3R)$ nontrivially. Let $\gamma_i$ be a connected subset of $C_i \cap \lp \cl{B}(p,3R) \backslash B(p,R/2)\rp$ that intersects both $B(p,R)$ and $Y\backslash \cl{B}(p,2R)$. Note that there is a positive distance between any two $\gamma_i, \gamma_j$ because there is a positive distance between $C_i \cap \cl{B}(p,3R)$ and $C_j \cap \cl{B}(p,3R)$. Now, let $0 < \e <1$ be small enough that $\e R$ is less than one fourth of the minimal distance between any two $\gamma_i, \gamma_j$. Note that along each $\gamma_i$, there is a discrete $\e R$-path from $B(p,R)$ to $Y\backslash \cl{B}(p,2R)$. Moreover, the $2\e R$-neighborhoods of these discrete paths do not intersect. Thus, any set of points $K$ that meets the $\e R$-neighborhood of each discrete $\e R$-path from $B(p,R)$ to $Y\backslash \cl{B}(p,2R)$ must contain a point in each of these neighborhoods. In particular, $\#K \geq k$. 

By Lemma \ref{weaktanuws}, we know that $Y$ has the discrete UWS property with constant $C'$ depending only on $C$ and the doubling constant of $X$. Thus, it has such a set $K$ with $\#K \leq C'$. This immediately gives $k \leq C'$, as desired.
\end{proof}

\section{Conformal dimension of some H\"older circles} \label{Holdersec}

Our discussions in this paper have focused mostly on metric spaces with properties similar to those found on boundaries of John domains. In this final section we give an example of a H\"older domain whose boundary is a Jordan curve with Ahlfors-regular conformal dimension equal to 2. The construction is straightforward and suggests that one should not hope to prove smaller upper bounds on the Ahlfors-regular conformal dimension of boundaries of H\"older domains in much generality. Instead, the conformal dimension is probably a more interesting quantity in the H\"older setting.

We begin with the boundary of the unit square in $\C$. Along the bottom edge, which we identify with the interval $[0,1]$, take points $x_k = 2^{-k}$ for $k \in \N$. Let $r_k = a^{-k}$, where $a \geq 10$ is fixed, and let $\omega_k = e^{i\pi/4k}$. For each $k$, we replace the interval $(x_k, x_k + r_k)$ with the polygonal arc $\alpha_k$ in the unit square which visits the points
$$\begin{aligned}
x_k,
&\hspace{4pt} x_k + r_k \omega_k^{2k}, 
\hspace{2pt}  x_k + 2r_k \omega_k^{2k},
\hspace{2pt} x_k + 2r_k \omega_k^{2k-1}, 
\hspace{2pt} x_k + r_k \omega_k^{2k-1}, 
\hspace{2pt} x_k + r_k \omega_k^{2k-2}, \\
&\hspace{2pt} x_k + 2r_k \omega_k^{2k-2},
\ldots \ldots,
\hspace{2pt} x_k + 2r_k \omega_k, 
\hspace{2pt} x_k + r_k \omega_k, 
\hspace{2pt} x_k + r_k
\end{aligned}$$
in this order. Informally, $\alpha_k$ forms $2k$ ``spokes" in the annulus $\cl{B}(x_k,2r_k) \backslash B(x_k,r_k)$ if we view the annulus as a bike wheel. Notice that each arc $\alpha_k$ lies in the ball $\cl{B}(x_k,2r_k)$, so none of these arcs interfere with one another. Thus, it is clear that the resulting set $X$ is a Jordan curve. 

It is also not difficult to see that the inner domain, $\Omega$, bounded by $X$ is a H\"older domain. Indeed, every point in $\Omega$ can be joined to the midpoint of the unit square by a John arc (with uniform constant), except for the points that lie between two consecutive spokes in some $\alpha_k$. If $z$ is one of these latter points, then we observe the following. First, there is a John arc in $\Omega$ from $z$ to a point $z_1$ which has $\dist(z, X) \leq \dist(z_1,X) \approx r_k/k$. Then there is an arc in $\Omega$ from $z_1$ to a point $z_2$, where $z_2$ is in the the uniform John region in $\Omega$, that has length at most $r_k$ and lies at distance $\gtrsim r_k/k$ from $X$. It is straightforward to estimate the quasi-hyperbolic length of the first segment by $C \log(1/\delta_\Omega(z))$ and the second segment by $Ck \approx \log(1/\delta_\Omega(z_1))$, where $C$ is uniform. Using a John arc (with uniform constant) from $z_2$ to the midpoint of the unit square, which has quasi-hyperbolic length at most $C \log(1/\delta_\Omega(z_2))$, we see that the quasi-hyperbolic distance from $z$ to the midpoint is $\lesssim \log (1/\delta_\Omega(z))$.

Note that $X$ has Hausdorff dimension 1, so $\Cdim(X) = 1$. We claim now that $\Cdim_{\AR}(X) = 2$. The estimate $\Cdim_{\AR}(X) \leq 2$ comes from the fact that $X$ is a subset of the plane \cite[Corollary 14.17]{Hein01}. For the other inequality, we verify a lower bound on combinatorial modulus. To this end, fix $1<p<2$. For each $k \in \N$, let $P_k$ be an $a^{-k}$-separated set in $X$ with $x_k \in P_k$. Fix $\lambda \geq 32$, and let $G_k$ be the graph with vertex set $P_k$, where $x,y \in P_k$ are joined by an edge if $B(x,\lambda a^{-k}) \cap B(y,\lambda a^{-k}) \neq \emptyset$. We should note that balls are taken to be in $X$, not in the plane.

Fix $k \in \N$ very large and let $m$ be the largest integer for which $100\lambda m \leq a^k$. Consider the family $\Gamma_k(x_m)$ of vertex paths in $G_{m+k}$ that join $P_{m+k} \cap B(x_m,a^{-m})$ to $P_{m+k} \cap (X \backslash \cl{B}(x_m, 2 a^{-m}))$. Let $\gamma_1,\ldots, \gamma_{2m}$ denote the spokes in $\alpha_m$, which are line segments from $\partial B(x_m,a^{-m})$ to $\partial B(x_m,2 a^{-m})$. Note that the distance between any two spokes is at least $a^{-m}/10m$ and, so also, at least $10\lambda a^{-m-k}$. For each $i$, there is a vertex path $\tilde{\gamma}_i$ in $\Gamma_k(x_m)$ consisting of vertices that lie on $\gamma_i$. The paths $\tilde{\gamma}_i$ are pairwise vertex disjoint and $\# \tilde{\gamma}_i \approx a^k$ with uniform constants.

Suppose that $\rho \colon P_{m+k} \rightarrow [0,\infty]$ is an admissible weight function for $\Gamma_k(x_m)$. Then for each $i$, we have
$$1 \leq \sum_{v \in \tilde{\gamma}_i} \rho(v) \leq \lp \# \tilde{\gamma}_i \rp^{(p-1)/p} \lp \sum_{v \in \tilde{\gamma}_i} \rho(v)^p \rp^{1/p} \lesssim a^{k(p-1)/p} \lp \sum_{v \in \tilde{\gamma}_i} \rho(v)^p \rp^{1/p},$$
so that
$$\sum_{v \in \tilde{\gamma}_i} \rho(v)^p \gtrsim a^{-k(p-1)}.$$
Consequently, we can estimate
$$\sum_{v \in P_{m+k}} \rho(v)^p \geq \sum_{i=1}^{2m} \sum_{v \in \tilde{\gamma}_i} \rho(v)^p \gtrsim m \cdot a^{-k(p-1)} \gtrsim a^k  \cdot a^{-k(p-1)} = a^{k(2-p)},$$
and therefore $\mod_p(\Gamma_k(x_m),G_{m+k}) \gtrsim a^{k(2-p)}$. Using the notation from before, we conclude that
$$M_p = \liminf_{k \rightarrow \infty} M_p(k) \gtrsim \liminf_{k \rightarrow \infty} a^{k(2-p)} = \infty,$$
so $Q_N \geq p$. As $p \in (1,2)$ was arbitrary, we obtain $\Cdim_{\AR}(X) = Q_N \geq 2$.

\begin{remark}
One could also establish $\Cdim_{\AR}(X) \geq 2$ in the following way, without explicit modulus estimates. Namely, let $\Cdim_{\textup{A}}(X)$ denote the conformal Assouad dimension of $X$, which is an \ti{a priori} lower bound for $\Cdim_{\AR}(X)$; see \cite[Section 2.2]{MT10} for definitions. It is known that the conformal Assouad dimension does not increase under taking weak tangents \cite[Proposition 6.1.7]{MT10}, and it is bounded below by the topological dimension. Now, observe that the sequence $(X, r_k^{-1}|\cdot|, x_k)$ gives a weak tangent $Y$ of $X$ that contains an isometric copy of the planar quarter annulus with inner radius 1 and outer radius 2. Thus,
$$\Cdim_{\AR}(X) \geq \Cdim_{\textup{A}}(X) \geq \Cdim_{\textup{A}}(Y) \geq \dim_{\textup{top}}(Y) \geq 2.$$
I thank the referee for pointing out this alternative argument.
\end{remark}

This example shows that the conformal dimension and the Ahlfors-regular conformal dimension can be very different, even for Jordan curves that bound a simply connected H\"older domain. In fact, we expect this to happen generically for H\"older domains that come from $\SLE$ process. More concretely, we ask the following question.

\begin{question} \label{sle}
Let $0<\kappa < 4$, and let $\gamma$ denote an $\SLE_{\kappa}$ trace stopped at time $t=1$. Is it true that, almost surely, $\Cdim_{\AR}(\gamma) =2$ and $\Cdim(\gamma) = 1$?
\end{question}

\end{document}